\documentclass[11pt,a4paper]{article}
\usepackage{amsmath,amssymb,latexsym,amsthm}
\usepackage{mathrsfs,xcolor}
\usepackage{mathtools}
\usepackage[english]{babel}
\usepackage[latin2]{inputenc}
\newcommand\ioto{\stackrel{\mathclap{\iota}}{\to}}
\newcommand\pito{\stackrel{\mathclap{\pi}}{\to}}
\DeclareMathOperator{\End}{End}
\DeclareMathOperator{\Aut}{Aut}
\theoremstyle{definition}
\newtheorem{theor}{Theorem}[section]
\newtheorem{lem}[theor]{Lemma}
\newtheorem{co}[theor]{Corollary}
\newtheorem{prop}[theor]{Proposition}
\newtheorem{definition}{Definition}[section]

\begin{document}
	
\title{Anti-commutative algebras and their groups of automorphisms}
\author{\'Agota Figula and P\'eter T. Nagy}
\footnotetext{This paper was supported by the National Research, Development and Innovation Office (NKFIH) Grant No. K132951. }
\footnotetext{Corresponding Author: P\'eter T. Nagy}	
\maketitle

\begin{abstract}
We determine normal forms of the multiplication of four-dimensional anti-commutative algebras over a field $\mathbb K$ of characteristic zero having an analogous family of flags of subalgebras as the four-dimensional non-Lie binary Lie algebras, and hence can be considered as the closest relatives of binary Lie algebras. 
These algebras are extensions of $\mathbb K$ by the 3-dimensional nilpotent Lie algebra and at the same time extensions of a two-dimensional Lie algebra by a two-dimensional abelian algebra. 
We describe their groups of automorphisms as extensions of a subgroup of the group of automorphisms of the three-dimensional nilpotent Lie algebra by $\mathbb K$. 
\end{abstract}

\maketitle

\section{Introduction}
The classification of tensors and the determination of their normal form is an old problem in mathematics (cf. Chapter 10 in \cite{landsberg}). Torsion tensors are defined in differential geometry as bilinear antisymmetric vector-valued maps on the tangent space of a manifold, the four-dimensional space-time manifolds with torsion tensor fields have important applications in the Einstein-Cartan theory of unified field theory (\cite{capsto}, \cite{petti}). In our next study, we consider such tensors over four-dimensional vector spaces interpreting them as anti-commutative multiplications, defining a kind of generalized Lie algebras. We find the normal forms of these multiplications and classify the algebras defined by these operations as well as their automorphism groups.\\
Various generalizations of Lie algebras occur naturally in many contexts in modern mathematics and physics.  Malcev algebras were introduced by A. Malcev in 1955 \cite{malcev} as tangent algebras of local analytic Moufang loops. In the same paper, the binary Lie algebra is defined as an anti-commutative algebra such that any subalgebra generated by two elements is a Lie algebra, hence these algebras are the tangent algebras of local analytic diassociative  loops (\cite{kuzmin3}). Identities of the variety of binary
Lie algebras were described by Gainov \cite{gainov0}. Note that every Lie algebra is a Malcev
algebra and every Malcev algebra is a binary Lie algebra.
The systematic study of Malcev
and binary Lie algebras began with the work of A. A. Sagle \cite{sagle}.
Classifications, extensions and degenerations of binary Lie algebras, Malcev algebras have been studied by many authors \cite{abdelwahab}, \cite{abdelwahab2}, \cite{elduque}, \cite{elduqueshestakov}, \cite{gainov}, \cite{gavrilov}, \cite{ismailov}, \cite{kaygorodov},  \cite{kuzmin1}, \cite{kuzmin2}. \\
One of the most natural and simple generalizations of binary Lie and Malcev algebras is the \emph{anti-commutative algebra} defined by an anti-symmetric bilinear operation on a vector space over a field. Several articles have recently appeared on the classification problems of such low-dimensional algebras (called also zeropotent algebras), with special attention to the nilpotent case (cf. \cite{kaykhrlop} and \cite{cedilnik}, \cite{kobayashi}, \cite{shirayanagi}).  Our paper is devoted to the study of a solvable class of $4$-dimensional anti-commutative algebras. The definition of the studied class of solvable anti-commutative algebras is suggested by the structure of 4-dimensional binary Lie algebras.
Binary Lie algebras of dimension $\le 4$ over a field $\mathbb K$ of characteristic different from $2$ have been described and the non-Lie binary Lie algebras of minimal dimension $4$ have been determined by \cite{gainov}. The case of characteristic $2$ of $\mathbb K$ is treated in \cite{kuzmin2}. These algebras are extensions of $\mathbb K$ by the 3-dimensional nilpotent Lie algebra and at the same time extensions of a two-dimensional Lie algebra by a two-dimensional abelian algebra. A family of flags of subalgebras defined by algebraic properties
can be associated to these binary Lie algebras.
Our purpose is to classify a class of $4$-dimensional anti-commutative algebras having an analogous family of flags of subalgebras and hence can be considered as the closest relatives of binary Lie algebras. We call these algebras $\mathcal{BL}_4$-algebras. A general method of investigation of isomorphism classes of algebras is to consider the orbits of the action of the general linear group via change of basis, the orbits under this action correspond to isomorphism classes (cf. \cite{BuSt}). Instead of considering the general linear group we limit our attention to the subgroup consisting of linear maps preserving the given associated families of flags of these algebras.\\
We summarize the basic concepts on binary Lie and Malcev algebras and on the automorphism group of the 3-dimensinal nilpotent Lie algebra in \S 2. The associated family of flags and the bases adapted to these flags, called distinguished bases, and the group of weak isomorphisms, preserving the associated family of flags, are introduced in \S 3. We find a matrix equation describing weak isomorphisms which are isomorphisms in \S 4. We use this matrix equation for the determination of normal forms of $\mathcal{BL}_4$-algebras in \S 5, which means that we select on each orbit elements of simple shape of the corresponding multiplication relations. \S 6
is devoted to the study of the isomorphism problems of $\mathcal{BL}_4$-algebras with multiplication of normal form. In \S 7 we give the classification of $\mathcal{BL}_4$-algebras using their multiplication of normal form. We introduce in \S 8 non-isomorphic groups which are candidates for groups of automorphisms of  $\mathcal{BL}_4$-algebras and investigate their properties. In \S 9 we use the matrix equation describing weak isomorphisms which are isomorphisms for the classification of groups of automorphisms of $\mathcal{BL}_4$-algebras.

	\section{Preliminaries}

We consider a (nonassociative) anti-commutative algebra $\mathfrak{g}$ over a field $\mathbb{K}$ of
characteristic zero.   The multiplication of the elements $x,y\in\mathfrak{g}$ will be denoted by $x\cdot y$. The multiplication symbol will be omitted whenever it does not cause confusion, the multiplication symbol $\cdot$ considered less binding than the juxtaposition, i.e. $xy\cdot u$ is a short form of $(x\cdot y)\cdot u$. The group of automorphisms of $\mathfrak{g}$ is denoted by $\Aut(\mathfrak{g})$, the algebra of linear maps of $\mathfrak{g}$ by $\End(\mathfrak{g})$. The left and the right multiplication maps $L_x,R_x:\mathfrak{g}\to \mathfrak{g}$ are defined by $L_x(t)= xt$, respectively, $R_x(t)= tx$,\; $x,t\in\mathfrak{g}$. The maps $L_x$, $R_x$ of $\mathfrak{g}$ differ only by sign, left and right ideals coincide, etc. The Jacobian $\mathcal{J}:\mathfrak{g}\times\mathfrak{g}\times\mathfrak{g}\to\mathfrak{g}$ is the map  $\mathcal{J}(x,y,z) = xy\cdot z+ zx\cdot y+ yz\cdot x$, $x, y,z\in\mathfrak{g}$. The commutator algebra $\mathfrak{g}^\prime = \mathfrak{g}\cdot\mathfrak{g}$ is an ideal, the second commutator algebra $\mathfrak{g}^{\prime\prime} = \mathfrak{g}^\prime\cdot\mathfrak{g}^\prime$ is a subalgebra in $\mathfrak{g}$.
Clearly, $\mathfrak{g}$ is a Lie algebra if and only if  $\mathcal{J}(\mathfrak{g}) =\{0\}$. An algebra $\mathfrak{g}$ is a \emph{binary Lie algebra} if the identity
\begin{equation}\label{binarylieiden}
\mathcal{J}(x,y,xy)=(y\cdot xy)x+(xy\cdot x)y =0
\end{equation}
holds for all $x, y \in \mathfrak{g}$. An algebra $\mathfrak{g}$ is a \emph{Malcev algebra} if for any $x, y, z \in \mathfrak{g}$  the identity
\begin{equation}\label{malceviden}
xy\cdot xz=(xy\cdot z)x+(yz\cdot x)x+(zx\cdot x)y  \nonumber 
\end{equation}
is satisfied.  
 An algebra $\mathfrak{g}$ is called decomposable if $\mathfrak{g}$ is the direct sum of subalgebras.\\
 An extension $\mathfrak{c}$ of an algebra $\mathfrak{b}$ by an algebra $\mathfrak{a}$ is a short exact sequence
 \[0\to \mathfrak{a}\ioto \mathfrak{c}\pito \mathfrak{b}\to 0,\]
 of algebras, where $\iota(\mathfrak{a})$ is an ideal of $\mathfrak{c}$ and $\pi$ induces an isomorphism of the factor algebra $\mathfrak{c}/\iota(\mathfrak{a})$ to $\mathfrak{b}$.  An extension $\mathfrak{c}$ is a semidirect sum of $\mathfrak{b}$  and $\mathfrak{a}$ if there exists a subalgebra $\widetilde{\mathfrak{b}}$ of $\mathfrak{c}$ such that $\pi(\widetilde{\mathfrak{b}})$ is isomorphic to $\mathfrak{b}$. \\
 Similarly, an extension $\mathcal{C}$ of a group $\mathcal{B}$  by a group $\mathcal{A}$ is a short exact sequence
 \[1\to \mathcal{A}\ioto\mathcal{C}\pito \mathcal{B}\to 1,\]
 where $\iota(\mathcal{A})$ is a normal subgroup of $ \mathcal{C}$ and $\pi:\mathcal{C}\to \mathcal{B}$ induces an isomorphism $\mathcal{C}/\iota( \mathcal{A})\to \mathcal{B}$, furthermore $\mathcal{C}$ is a semidirect product of $\mathcal{B}$ and $\mathcal{A}$ if there is a subgroup $\widetilde{\mathcal{B}}$ of $\mathcal{C}$ such that $\pi(\widetilde{\mathcal{B}})$ is isomorphic to $\mathcal{B}$.
 \\
Let two bases $\mathcal{E}=\{e_0,e_1,\dots,e_{n-1}\}$, $\hat{\mathcal{E}}=\{\hat{e}_0,\hat{e}_1,\dots,\hat{e}_{n-1}\}$ be given in the algebra $\mathfrak{g}$. The vector space $\mathbb K^n$ can be endowed with algebra structures $\mathbb K^n(\mathcal{E})$ and $\mathbb K^n(\hat{\mathcal{E}})$, so that the linear coordinate maps $\phi_{\mathcal{E}}:\mathfrak{g}\to\mathbb K^n$ and $\phi_{\hat{\mathcal{E}}}:\mathfrak{g}\to\mathbb K^n$ become isomorphisms of algebras. The composition $\phi_{\hat{\mathcal{E}}}\circ\phi_{\mathcal{E}}^{-1}:\mathbb K^n(\mathcal{E})\to\mathbb K^n(\hat{\mathcal{E}})$ is an isomorphism between algebras corresponding to the change of the bases $\mathcal{E}$ and $\hat{\mathcal{E}}$. \\[1ex]
{\bf Theorem}  (Gainov)
\emph{The minimal dimension of non-Lie binary Lie algebras is $4$, they are given on $\mathbb K^4$ up to isomorphism by the mutually non-isomorphic anti-commutative algebras with non-vanishing multiplications
\begin{equation}\label{gainov}\begin{split}
&e_1e_2 = e_3,\; e_0e_3 = e_3, \\
&e_1e_2 = e_3,\;  e_0e_1 = e_1,\; e_0e_2 = e_2,\; e_0e_3 =\mu e_3,\quad \mu\in \mathbb K.
	\end{split}  \end{equation}
The algebra belonging to the parameter $\mu\in \mathbb K$ is a Lie algebra for $\mu = 2$, a non-Lie Malcev algebra for $\mu = -1$, and non-Malcev binary Lie algebra if $\mu\notin \{-1,2\}$.}\\[1ex]
In the following we will classify the four-dimensional anti-commutative algebras having the same ideal structure as these binary Lie algebras \eqref{gainov}.

\subsection*{$3$-dimensional non-abelian nilpotent Lie algebra}
Denote  $\mathfrak{n}$ the $3$-dimensional non-abelian nilpotent Lie algebra on  $\mathbb K^3$, the ideal $\mathfrak{n}^\prime$ has dimension one.
The group $\Aut(\mathfrak{n})$ of automorphisms of $\mathfrak{n}$ preserves the commutator $\mathfrak{n}^\prime$, hence $\Aut(\mathfrak{n})$ is represented by the group of matrices
\begin{equation}\label{indaut}\begin{bmatrix}
		p^1 & q^1 & 0\\
		p^2 & q^2 & 0\\
		p^3 & q^3 & p^1q^2-p^2q^1
	\end{bmatrix}, \quad  (p^1q^2-p^2q^1) \neq 0,\end{equation}
with respect to a basis $\{e_1,e_2,e_3\}$ satisfying $e_1e_2 = e_3$. Using decomposition of matrices \eqref{indaut} for the blocks $\begin{bmatrix}
	p^1 & q^1\\
	p^2 & q^2
\end{bmatrix}$,  ${\bf a}^t=\begin{bmatrix}
	p^3 & q^3
\end{bmatrix}$, $|A| =p^1q^2-p^2q^1$, the multiplication can be expressed as
$$\begin{bmatrix}
	A & 0\\
	{\bf a}^t & |A|
\end{bmatrix}\begin{bmatrix}
	B & 0\\
	{\bf b}^t & |B|
\end{bmatrix}=\begin{bmatrix}
	AB & 0\\
	{\bf a}^t B + |A|{\bf b}^t& |AB|
\end{bmatrix}.$$
Denote $\mathbb K^*$ the multiplicative group of $\mathbb K$ and consider the bijective map $$\Delta:\mathbb K^*\times SL(2,\mathbb K)\times \mathbb K^2\to\Aut(\mathfrak{n}),\quad  (\xi,X,{\bf x}^t)\mapsto\begin{bmatrix}
	\xi X & 0\\
	\xi {\bf x}^t & \xi
\end{bmatrix}$$
and its inverse
$$\Delta^{-1}:\Aut(\mathfrak{n})\to\mathbb K^*\times SL(2,\mathbb K)\times \mathbb K^2,\quad  \begin{bmatrix}
	A & 0\\
	{\bf a}^t & |A|
\end{bmatrix}\mapsto (|A|,\frac{A}{|A|},\frac{{\bf a}^t}{|A|}).$$
We introduce on $\mathbb K^*\times SL(2,\mathbb K)\times \mathbb K^2$ the multiplication
$$(\xi,X,{\bf x}^t)\cdot (\eta,Y,{\bf y}^t)=(\xi\eta,XY,{\bf x}^tY+{\bf y}^t),$$
then $\mathbb K^*\times SL(2,\mathbb K)\times \mathbb K^2$ becomes a group isomorphic to $\Aut(\mathfrak{n})$. It follows
\begin{lem}
	The group $\Aut(\mathfrak{n})$ is the direct product $\mathbb K^*\times SL(2,\mathbb K)\ltimes\mathbb K^2$ of $\mathbb K^*$ with the special affine group  $SA(2,\mathbb K)$ isomorphic to the semidirect product   $SA(2,\mathbb K)\cong SL(2,\mathbb K)\ltimes\mathbb K^2$.
\end{lem}

\section{Weakly isomorphic algebras, $\mathcal{BL}_4$-algebras}

	We consider the following class of anti-commutative algebras:
\begin{definition}\label{M4algebradef}
	A $4$-dimensional anti-commutative algebra $\mathfrak{c}$ is called \emph{$\mathcal{BL}_4$-algebra} if
	\begin{enumerate}
		\item[(i)] $\mathfrak{c}$ is non-decomposable and has an ideal $\mathfrak{p}$ isomorphic to $\mathfrak{n}$ satisfying  $\mathfrak{c}' = \mathfrak{p}$ or $\mathfrak{c}' = \mathfrak{p}'$,
		\item[(ii)]  $\mathfrak{c}$ contains a $2$-dimensional Lie subalgebra $\mathfrak{a}$ with multiplication $e_0e_1 = x_1^1 e_1$, $x_1^1\in\mathbb{K}$ with respect to a basis $\{e_0,e_1\}$ of $\mathfrak{a}$, and a $2$-dimensional abelian ideal $\mathfrak{b}$ such that $\mathfrak{c}$ is the semidirect sum of  $\mathfrak{a}$ and  $\mathfrak{b}$ and $\mathfrak{p}'\subset\mathfrak{b}\subset\mathfrak{p}$.
	\end{enumerate}
\end{definition}
\begin{lem}\label{semidirect} The ideal
	$\mathfrak{p}$ and the subspace $\mathfrak{p}'$ are preserved by all automorphisms of $\mathfrak{c}$.
\end{lem}
\begin{proof} $\mathfrak{c}'$ is clearly a characteristic ideal. If $\mathfrak{c}' = \mathfrak{p}$ then $\mathfrak{p}$ and $\mathfrak{p}'$ are preserved by automorphisms. If $\mathfrak{c}' = \mathfrak{p}'$ then $\mathfrak{p}$ is contained in the subalgebra $\hat{\mathfrak{p}} =\{x\in\mathfrak{c}; xy=0\;\text{for any}\; y\in\mathfrak{p}'\}$. In the case $\hat{\mathfrak{p}}= \mathfrak{c}$ the non-decomposable algebra $\mathfrak{c}$ would be a direct sum of $\mathfrak{p}$ with a one-dimensional subalgebra, which is a contradiction. Hence $\mathfrak{p}=\hat{\mathfrak{p}}$ and the assertion is proved.
\end{proof}
\begin{definition} A sequence of subspaces $\mathfrak{p}^{\prime}\subset \mathfrak{b} \subset \mathfrak{p} \subset \mathfrak{c}$ is called \emph{adapted flag} of a $\mathcal{BL}_4$-algebra if  $\mathfrak{a}$ is a subalgebra, $\mathfrak{b}$ is an ideal of $\mathfrak{c}$, satisfying the condition (ii) in Definition 	\ref{M4algebradef}. \\
	A basis $\{e_0,e_1,e_2,e_3\}$ of $\mathfrak{c}$ is called \emph{distinguished} if for an adapted flag $\mathfrak{p}^{\prime}\subset \mathfrak{b} \subset \mathfrak{p} \subset \mathfrak{c}$ one has \[
	e_3\in\mathfrak{p}^{\prime},\; e_2\in\mathfrak{b}\setminus\mathfrak{p}^{\prime},\;  e_1\in\mathfrak{p}\setminus\mathfrak{b},\;  e_0\in\mathfrak{c}\setminus\mathfrak{p}.\]
\end{definition}
\begin{lem}
	A basis
	$\{e_0,e_1,e_2,e_3\}$ of the anti-commutative $\mathcal{BL}_4$-algebra $\mathfrak{c}$ is a distinguished basis if and only if
	\begin{enumerate}
		\item[(a)] $\{e_1,e_2,e_3\}$ is a basis of the non-abelian nilpotent ideal $\mathfrak{p}$ such that $e_3\in\mathfrak{p}^{\prime}$,
		\item[(b)] the matrix of the left multiplication map $L_0 =L_{e_0}\in \End(\mathfrak{p})$ with respect to the basis  $\{e_0,e_1,e_2,e_3\}$ has the form
		\begin{equation}\label{x1234} L_0 = \begin{bmatrix}
				x_1^1 & 0 & 0 \\
				0 & x_{2}^2 & x_3^2\\
				0 & x_2^3 & x_3^3
			\end{bmatrix}\neq \begin{bmatrix}
				0 & 0 & 0 \\
				0 & 0 & 0\\
				0 & 0 & 0
			\end{bmatrix}.\end{equation}
	\end{enumerate}
\end{lem}
\begin{proof}
	The form \eqref{x1234} of the matrix $L_0$ means that $$e_0e_1 = x_1^1 e_1,\; x_1^1\in\mathbb{K}\quad\text{and}\quad L_0(\mathbb K e_2+\mathbb K e_3)\subseteq(\mathbb K e_2+\mathbb K e_3).$$ Since $e_1e_2\in\mathbb K e_3$ and $e_1e_3=0$, the subspace $\mathbb K e_2+\mathbb K e_3$ is an abelian ideal. Consequently, the conditions (a) and (b) imply that $\{e_0,e_1,e_2,e_3\}$ is a distinguished basis. The converse statement is clear.
\end{proof}

\begin{definition}
A linear map $\alpha:\mathfrak{c}\to{\mathfrak{c}^{\ast}}$ between  $\mathcal{BL}_4$-algebras is called a \emph{weak isomorphism} if $\alpha$ maps distinguished bases of $\mathfrak{c}$ to distinguished bases of ${\mathfrak{c}^{\ast}}$. If there is a weak isomorphism $\alpha:\mathfrak{c}\to{\mathfrak{c}^{\ast}}$ then the $\mathcal{BL}_4$-algebras $\mathfrak{c}$ and ${\mathfrak{c}^{\ast}}$ are called \emph{weakly isomorphic}.
\end{definition}
It follows immediately from the definition:
\begin{lem} \begin{enumerate}
		\item[(i)] For any distinguished bases $\{e_0,e_1,e_2,e_3\}$ of $\mathfrak{c}$ and  $\{\hat{e}_0,\hat{e}_1,\hat{e}_2,\hat{e}_3\}$ of $\mathfrak{c}^{\ast}$ there is a unique weak isomorphism  $\alpha:\mathfrak{c}\to{\mathfrak{c}^{\ast}}$ with $\alpha(e_i) = \hat{e}_i$, $i=0,1,2,3$.
		\item[(ii)]  A linear map  $\alpha:\mathfrak{c}\to\mathfrak{c}^{\ast}$ is a weak isomorphism if and only if its matrix $M_{\alpha}$ with respect to the given bases has the form
	\begin{equation}\label{Amatrixaut4} M_{\alpha} = \begin{bmatrix}
			u^0 &0 & 0 & 0 \\
			u^1 &p^1 & q^1 & 0 \\
			u^2 &p^2 & q^2 & 0 \\
			u^3 &p^3 & q^3 & r^3
		\end{bmatrix},\; |M_{\alpha}|=u^0(p^1 q^2- p^2 q^1) r^3  \neq 0,\; u^0, r^3, u^i, p^i, q^i \in\mathbb{K}, i = 1,2,3.\end{equation}
	\end{enumerate}
\end{lem}
\begin{lem} The group $\Aut^\text{w}(\mathfrak{c})$ of weak automorphisms of an $\mathcal{BL}_4$-algebra $\mathfrak{c}$ acts simply transitively on the set of distinguished bases of $\mathfrak{c}$. The automorphism group $\Aut(\mathfrak{c})$ of $\mathfrak{c}$ is a subgroup of $\Aut^\text{w}(\mathfrak{c})$.
\end{lem}

\section{Orbits of multiplications}

Let $\mathbb{K}^4$ be the $4$-dimensional vector space over the field $\mathbb{K}$  and $\mathcal{E} =\{e_0,e_1,e_2,e_3\}$ its canonical basis. A $\mathcal{BL}_4$-algebra on $\mathbb{K}^4$ with distinguished basis $\mathcal{E}$ is determined by the multiplication of the basis vectors, i.e. by the system of structure constants $c_{ij}^k$ in the equations $e_ie_j = \sum_{0}^{3}c_{ij}^ke_k$. These constants have to fulfill the equations
$$c_{ij}^k = -c_{ji}^k,\; c_{12}^3=\xi^3,\; c_{0j}^k = -c_{j0}^k = x_j^k,\quad j,k = 1,2,3,\;\text{with}\; \{x_j^k\}= \begin{bmatrix}
	x_{1}^1 & 0 & 0\\	
	0 & x_{2}^2 & x_3^2\\
    0 & x_2^3 & x_3^3
\end{bmatrix},$$
and the constants other than these are zero.
Conversely, for any triple $\{x_1^1,\xi^3,X\}$, $x_1^1,\xi^3 \in\mathbb{K}$, $X = \begin{bmatrix}
	x_{2}^2 & x_3^2\\
	x_2^3 & x_3^3
\end{bmatrix}\in\mbox{End}(\mathbb{K}^2)$, the system of structure constants satisfying $c_{ij}^k = -c_{ji}^k$, $c_{01}^1=x_1^1$, $c_{12}^3=\xi^3$, $c_{0j}^k = -c_{j0}^k = x_j^k$, $j,k = 2,3,$ determines an anti-commutative algebra, which will be denoted by $\mathfrak{c}(x_1^1,\xi^3,X)$.
\begin{lem} \label{bl4algebra} An anti-commutative algebra $\mathfrak{c}(x_1^1,\xi^3,X)$, $x_1^1,\xi^3\in\mathbb{K}$, $X\in\mbox{End}(\mathbb{K}^2)$ is a $\mathcal{BL}_4$-algebra if and only if one of the following conditions is satisfied:
	\begin{enumerate}
		\item $x_1^1\neq 0$ and the first row of $X$ is non-vanishing,
		\item $x_1^1=0$, the first row of $X$ is vanishing and the second row of $X$ is non-vanishing.
	\end{enumerate}
\end{lem}
\begin{proof} If $x_1^1\neq 0$ and the first row of $X$ is vanishing, then $\mathfrak{p}'\subsetneq \mathfrak{c}'\subsetneq \mathfrak{p}$ and hence $\mathfrak{c}(x_1^1,\xi^3,X)$ is not a $\mathcal{BL}_4$-algebra. If $x_1^1\neq 0$ and the first row of $X$ is non-vanishing, then $L_{e_0}\mathfrak{p}+\mathfrak{p}' = \mathfrak{p}$, hence $\mathfrak{c}' = \mathfrak{p}$. For $x_1^1 = 0$ one has $\mathfrak{c}'\subsetneq \mathfrak{p}$, consequently $\mathfrak{c}(x_1^1,\xi^3,X)$ is a $\mathcal{BL}_4$-algebra if and only if $ \mathfrak{c}'= \mathfrak{p}'$, or equivalently, first row of $X$ is vanishing.
\end{proof}
Clearly, the set of distinguished bases and the group  $\Aut^\text{w}(\mathcal{E})$ of weak-automorphisms of a $\mathcal{BL}_4$-algebra is independent of the parameters $\{x_1^1,\xi^3,X\}$, it is related only to the sequence of subspaces $\mathfrak{c}\supset \mathfrak{p}\supset\mathfrak{b}\supset\mathfrak{p}'$.\\
Let $\{c_{ij}^k(X)\}$ be the system of structure constants of $\mathcal{BL}_4(X)$ and denote $\gamma(\mathcal{E})$ the set of systems of structure constants of multiplications of $\mathcal{BL}_4$-algebras $\mathcal{BL}_4(X)\in\mathcal{BL}_4(\mathcal{E})$. The group $\Aut^\text{w}(\mathcal{E})$ acts on $\gamma(\mathcal{E})$ by
\begin{equation}\label{actionmu4}\gamma_X(x, y) \mapsto M_{\alpha}^{-1}\gamma_X(M_{\alpha} x,M_{\alpha} y),\quad M_{\alpha} \in\Aut^\text{w}(\mathcal{E}),\quad x, y\in\mathbb{K}^4,\end{equation}
where $\gamma_X$ denotes the multiplication on $\mathbb{K}^4$ determined by the  system $\{c_{ij}^k(X)\}$ of structure constants.
\begin{lem}
	Consider the set $\Gamma$ of all possible $\mathcal{BL}_4$-algebra multiplications $\gamma$ on $\mathbb{K}^4$ having the canonical basis $\mathcal{E} =\{e_0,e_1,e_2,e_3\}$ as a distinguished basis. The
	orbits of the action \eqref{actionmu4} of the group $\Aut^\text{w}(\mathcal{E})$ on $\Gamma$ correspond to isomorphism classes of 4-dimensional $\mathcal{BL}_4$-algebras.
\end{lem}
\begin{proof} Let $\{e_0,e_1,e_2,e_3\}$ and  $\{\hat{e}_0,\hat{e}_1,\hat{e}_2,\hat{e}_3\}$ be distinguished bases of the isomorphic $\mathcal{BL}_4$-algebras $\mathfrak{c}$ and ${\mathfrak{c}^{\ast}}$, respectively. We identify ${\mathfrak{c}^{\ast}}$ with $\mathfrak{c}$  with an isomorphism ${\mathfrak{c}^{\ast}}\to\mathfrak{c}$. Since an isomorphism preserves the subspaces $\mathfrak{p}$, $\mathfrak{p}^{\prime}$,  the identification yields that $\{\hat{e}_0,\hat{e}_1,\hat{e}_2,\hat{e}_3\}$ is also a distinguished basis of $\mathfrak{c}$. It follows that the linear map $\alpha$ defined by $\alpha(e_i) = \hat{e}_i$, $i = 0,1,2,3$  is an isomorphism between the multiplications with respect to the bases  $\{e_0,e_1,e_2,e_3\}$ and  $\{\hat{e}_0,\hat{e}_1,\hat{e}_2,\hat{e}_3\}$, respectively.
\end{proof}
In the following we use the Einstein convention, that is for each pair of indices occuring once in an upper and once in a lower position in a term we do summation.
\begin{lem} \label{blxialgebra} Any $\mathcal{BL}_4$-algebra $\mathfrak{c}(x_1^1,\xi^3,X)$, $x_1^1,\xi^3\in\mathbb{K}$, $X\in\mbox{End}(\mathbb{K}^2)$ is isomorphic to a $\mathcal{BL}_4$-algebra $\mathfrak{c}(x_1^1,1,X)$.
\end{lem}
\begin{proof}
		According to \eqref{actionmu4} the structure constants
	$\{c_{ij}^k\}$ and $\{\hat{c}_{ij}^k\}$ determine isomorphic $\mathcal{BL}_4$-algebras if and only if there is a matrix $M_{\alpha}= \{a^i_{j}\}$ of the form \eqref{Amatrixaut4} satisfying
	\begin{equation}\label{cijklaij4} a_j^i\hat{c}_{kl}^j = c_{pq}^ia_k^pa_l^q, \quad i,j,k,l,p,q = 0,1,2,3.\end{equation}
Since $a_j^0 =u^0 \delta_j^0$, $c_{00}^i=\hat{c}_{00}^i =c_{kl}^0=\hat{c}_{kl}^0=0$, $a_0^i=u^i$, $i,j,k,l = 0,1,2,3$ the equations \eqref{cijklaij4} give only the relations
		\begin{equation}\label{cijxjk4}\begin{split} &a_3^i\hat{c}_{12}^3 =\delta_3^ir^3 \hat{\xi}^3 = c_{pq}^ia_1^pa_2^q=\delta_3^i \xi^3( p^1q^2-p^2q^1)\\
				 &a_j^i\hat{c}_{0l}^j = c_{pq}^iu^pa_l^q  = c_{0j}^iu^0a_l^j + c_{jh}^iu^ja_l^h,\end{split}\end{equation}
		where $i,j,h,l=1,2,3$, $p,q=0,1,2,3$. The second equations of (\ref{cijxjk4}) do not contain the parameter $\hat{\xi}^3$, hence we can get $\hat{\xi}^3=1$ with the replacement
			$r^3 = \xi^3( p^1q^2-p^2q^1)$. The matrix $M_{\alpha}=\begin{bmatrix}
			1 &0 & 0 & 0 \\
			0 & 1 & 0 & 0 \\
			0 & 0 & \xi^3( p^1q^2-p^2q^1) & 0 \\
			0 & 0 & 0 & \xi^3( p^1q^2-p^2q^1)
		\end{bmatrix}$ gives an isomorphism requisted in the assertion.
\end{proof}
This lemma shows that for the investigation of isomorphism classes of $\mathcal{BL}_4$-algebras it is sufficient to restrict our consideration to algebras $\mathfrak{c}(x_1^1,1,X)$, $\xi^3=1$,  $x_1^1\in\mathbb{K}$, $X\in\mbox{End}(\mathbb{K}^2)$. We will use the notation $\mathfrak{c}(x_1^1,X)=\mathfrak{c}(x_1^1,1,X)$.
We obtain that the matrix $M_{\alpha}$ of a weak isomorphism $\alpha:\mathfrak{c}\to\mathfrak{c}^{\ast}$ preserving the equation $e_1e_2=e_3$ has the form
\begin{equation}\label{Amatrixpaut} M_{\alpha} = \begin{bmatrix}
		u^0 &0 & 0 & 0 \\
		u^1 &p^1 & q^1 & 0 \\
		u^2 &p^2 & q^2 & 0 \\
		u^3 &p^3 & q^3 & p^1q^2-p^2q^1
	\end{bmatrix},\; p^1q^2-p^2q^1 \neq 0,\; 0\neq u^0, u^i, p^i, q^i \in\mathbb{K},\; i = 1,2,3.\end{equation}
\begin{theor}  Let $\{e_0,e_1,e_2,e_3\}$ and  $\{\hat{e}_0,\hat{e}_1,\hat{e}_2,\hat{e}_3\}$ be distinguished bases of the $\mathcal{BL}_4$-algebras $\mathfrak{c}(x_1^1,X)$ and $\hat{\mathfrak{c}}(\hat{x}_1^1,\hat{X})$, respectively, satisfying
	$$e_0e_j = x_j^ke_k,\quad \hat{e}_0\hat{e}_j = \hat{x}_j^k\hat{e}_k,\quad j,k={ 2, 3}.$$
	The weak isomorphism $\alpha:e_i\mapsto \hat{e}_i$, $i=0,1,2,3$,  determined by the  matrix \eqref{Amatrixpaut} is an isomorphism if and only if
	\begin{equation}\label{isommatrixeq}\begin{split}&\begin{bmatrix}
			p^1\hat{x}_1^1 & q^1\hat{x}_2^2 & q^1\hat{x}_3^2\\
			p^2\hat{x}_1^1 & q^2\hat{x}_2^2 & q^2\hat{x}_3^2\\
			p^3\hat{x}_1^1 & q^3\hat{x}_2^2 +(p^1q^2-p^2q^1)\hat{x}_2^3& q^3\hat{x}_3^2 +(p^1q^2-p^2q^1)\hat{x}_3^3
		\end{bmatrix} =\\=&\begin{bmatrix}	
			u^0p^1x_1^1 & u^0q^1x_1^1 & 0\\
			u^0x_j^2p^j & u^0x_j^2q^j & u^0x_3^2(p^1q^2-p^2q^1)\\
			u^0x_j^3p^j+u^1p^2 - u^2p^1 &  u^0x_j^3q^j +u^1q^2 - u^2q^1 & u^0x_3^3(p^1q^2-p^2q^1)
		\end{bmatrix}.\end{split}\end{equation}
\end{theor}
\begin{proof}
	According to \eqref{cijxjk4} the structure constants
	$\{c_{ij}^k\}$ and $\{\hat{c}_{ij}^k\}$ determine isomorphic $\mathcal{BL}_4$-algebras if and only if there is a matrix $M_{\alpha}= \{a^i_{j}\}$ of the form \eqref{Amatrixaut4} satisfying
\[r^3 = p^1q^2-p^2q^1\]
and
		\begin{equation}\label{cijkxjk4} a_j^i\hat{c}_{0l}^j = a_j^i\hat{x}_{l}^j = c_{pq}^iu^pa_l^q  = c_{0j}^iu^0a_l^j + c_{jh}^iu^ja_l^h = x_{j}^iu^0a_l^j + \delta_3^i(u^1a_l^2 - u^2a_l^1), \nonumber \end{equation}
		where $i,j,h,l=1,2,3$, $p,q=0,1,2,3$. Equivalently,  we have
\begin{equation}\label{aijxdelta}
a_j^i\hat{x}_{l}^j = x_{j}^iu^0a_l^j + \delta_3^i(u^1a_l^2 - u^2a_l^1),\quad i,j,l=1,2,3.\end{equation}
Using the matrix blocks $\{a_j^i\}$, $\{x_j^i\}$, $\{\hat{x}_j^i\}$, $i,j,l=1,2,3$, of the form
$$\{a_j^i\} = \begin{bmatrix}
p^1 & q^1 & 0 \\
p^2 & q^2 & 0 \\
p^3 & q^3 & p^1q^2-p^2q^1
\end{bmatrix},\; \{x_j^i\} =\begin{bmatrix}
x_1^1 & 0 & 0 \\
0 & x_{2}^2 & x_{3}^2\\
0 & x_{2}^3 & x_{3}^3
\end{bmatrix},\; \{\hat{x}_j^i\} =\begin{bmatrix}
\hat{x}_1^1 & 0 & 0 \\
0 & \hat{x}_{2}^2 & \hat{x}_{3}^2\\
0 & \hat{x}_{2}^3 & \hat{x}_{3}^3
\end{bmatrix}$$
we obtain the assertion from the equations \eqref{aijxdelta}.
\end{proof}
\begin{co} The $\mathcal{BL}_4$-algebras $\mathfrak{c}(x_1^1,X)$ and $\hat{\mathfrak{c}}(\hat{x}_1^1,\hat{X})$ are isomorphic if and only if there is a weak isomorphism with matrix \eqref{Amatrixpaut} satisfying the equation \eqref{isommatrixeq}.
\end{co}

\section{Normal forms}

In this section we prove that for any $\mathcal{BL}_4$-algebra the left translation $L_{e_0}$ can be reduced to a normal form by an isomorphism.
\begin{definition}
	We say that \begin{enumerate}
		\item a weak isomorphism between $\mathcal{BL}_4$-algebras determined by the matrix \eqref{Amatrixpaut} is
		\begin{enumerate} 	\item \emph{direct} if $p^1q^2\neq 0$ and $p^2=q^1 =0$,
		\item \emph{opposite} if $p^2q^1 \neq0$ and $p^1=q^2 = 0$,
		\end{enumerate}
		\item  the $\mathcal{BL}_4$-algebras $\mathfrak{c}(x_1^1,X)$ and $\hat{\mathfrak{c}}(\hat{x}_1^1,\hat{X})$ are \emph{directly}  \emph{(oppositely)} isomorphic if the equation \eqref{isommatrixeq} is satisfied by a direct (opposite) weak isomorphism.
	\end{enumerate}
\end{definition}
\begin{prop}\label{diagonal}
Any $\mathcal{BL}_4$-algebra $\mathfrak{c}$ with $L_{e_0} = \begin{bmatrix}
		x_1^1 & 0 & 0 \\
		0 & x_{2}^2 & 0\\
		0 & x_{2}^3 & x_{3}^3
	\end{bmatrix}$ is directly isomorphic to the $\mathcal{BL}_4$-algebra $\hat{\mathfrak{c}}$ with  $$L_{e_0} = \begin{bmatrix}
	1 & 0 & 0 \\
	0 & \frac{x_2^2}{x_1^1} & 0\\
	0 & 0 & \frac{x_{3}^3}{x_1^1}
\end{bmatrix},\; \text{if}\; x_1^1\neq 0\quad \text{or
}\quad L_{e_0} = \begin{bmatrix}
0 & 0 & 0 \\
0 & 0 & 0\\
0 & 0 & 1
\end{bmatrix},\; \text{if}\; x_1^1 = x_2^2 =0.$$
\end{prop}
\begin{proof}
Assume  $p^2=p^3 = q^1=q^3=u^2 =0$. Then we get from \eqref{isommatrixeq} the matrix equation		
		\[\begin{bmatrix}
			p^1\hat{x}_1^1 & 0 & 0\\
			0 & q^2\hat{x}_2^2 & 0\\
			0 & 0 & p^1q^2 \hat{x}_3^3\end{bmatrix} = \begin{bmatrix}	
			u^0p^1x_1^1 & 0 & 0\\
			0 & u^0x_2^2q^2 & 0\\
			0 &  u^0x_2^3q^2 + u^1q^2 & u^0x_3^3p^1q^2
		\end{bmatrix}.\]	
 If $x_1^1 =0$ then we have $\mathfrak{c}' = \mathfrak{p}'$, hence $x_2^2 =0$, $x_3^3\neq 0$. With the  matrices
\[ \text{if}\quad x_1^1 =0,\quad M_{\alpha}=\begin{bmatrix}
	\frac{1}{x_3^3} &0 & 0 & 0 \\
	-\frac{x_2^3}{x_3^3} & 1 & 0 & 0 \\
	0 & 0 & 1 & 0 \\
	0 & 0 & 0 & 1
\end{bmatrix}\quad\text{and if}\quad x_1^1\neq 0,
 \quad M_{\alpha}=\begin{bmatrix}
 	\frac{1}{x_1^1} &0 & 0 & 0 \\
 	-\frac{x_2^3}{x_1^1} & 1 & 0 & 0 \\
 	0 & 0 & 1 & 0 \\
 	0 & 0 & 0 & 1
 \end{bmatrix},\]
 we obtain $\hat{x}_1^1=\hat{x}_2^2=0$, $\hat{x}_3^3=1$, respectively, $\hat{x}_1^1= 1$, $\hat{x}_2^2=\frac{x_2^2}{x_1^1}$, $\hat{x}_3^3=\frac{x_{3}^3}{x_1^1}$.
Thus the assertion is proved.
\end{proof}
\begin{prop}\label{1diagonal}
Any $\mathcal{BL}_4$-algebra $\mathfrak{c}$ with $L_{e_0} = \begin{bmatrix}
		x_1^1 & 0 & 0 \\
		0 & x_2^2 & x_3^2 \\
		0 & x_2^3 & x_3^3
	\end{bmatrix}$, $x_1^1x_3^2 \neq 0$, is directly isomorphic to the $\mathcal{BL}_4$-algebra $\hat{\mathfrak{c}}$ determined by $L_{e_0} =
	\begin{bmatrix}
		1 & 0 & 0 \\
		0 & \frac{x_2^2+x_3^3}{x_1^1} & 1 \\	
		0 & 0 & 0
	\end{bmatrix}$.
\end{prop}
\begin{proof} We assume $\hat{x}_1^1 = 1$, $\hat{x}_2^2 = \frac{x_2^2+x_3^3}{x_1^1}$, $\hat{x}_3^2 =1$. Hence the replacement of the matrix
	$M_{\alpha}=\begin{bmatrix}
		\frac{1}{x_1^1} &0 & 0 & 0 \\
		\frac{x_3^3x_2^2}{x_3^2x_1^1}-\frac{x_2^3}{x_1^1} & \frac{x_1^1}{x_3^2} & 0 & 0 \\
		0 & 0 & x_3^2 & 0 \\
		0 & 0 & x_3^3 & x_1^1
	\end{bmatrix}$
	into the equation \eqref{isommatrixeq} yields the assertion.  	
\end{proof}

\section{Isomorphisms}

\begin{prop}\label{eqdiagonal}
The $\mathcal{BL}_4$-algebras with $L_{e_0} = \begin{bmatrix}
		1 & 0 & 0 \\
		0 & x_{2}^2 & 0\\
		0 & 0 & x_{3}^3
	\end{bmatrix}$ and $L_{e_0} = \begin{bmatrix}
		1 & 0 & 0 \\
		0 & \hat{x}_{2}^2 & 0\\
		0 & 0 & \hat{x}_{3}^3
	\end{bmatrix}$ are isomorphic if and only if they are:
\begin{enumerate}
	\item[(a)] directly isomorphic and $\hat{x}_2^2=x_2^2, \; \hat{x}_3^3=x_3^3$,
    \item[(b)] oppositely isomorphic and  $\hat{x}_2^2 =\frac{1}{x_2^2}, \quad \hat{x}_3^3=\frac{x_3^3}{x_2^2}$.
\end{enumerate}
\end{prop}
\begin{proof}
The matrix equation
\[\begin{split}&\begin{bmatrix}
		p^1 & q^1\hat{x}_2^2 & 0\\
		p^2 & q^2\hat{x}_2^2 & 0\\
		p^3 & q^3\hat{x}_2^2 & (p^1q^2-p^2q^1)\hat{x}_3^3
	\end{bmatrix} =\\=&\begin{bmatrix}	
		u^0p^1 & u^0q^1 & 0\\
		u^0x_2^2p^2 & u^0x_2^2q^2 & 0\\
		u^0x_3^3p^3+u^1p^2 - u^2p^1 & u^0x_3^3q^3 +u^1q^2 - u^2q^1 & u^0x_3^3(p^1q^2-p^2q^1)
	\end{bmatrix}\end{split}\]
gives the equations \begin{equation}\label{felsotri}\begin{split}
	(u^0-1)p^1=0,\; (u^0x_2^2-1)p^2=0,\;  (\hat{x}_2^2-u^0)q^1=0,\; (\hat{x}_2^2-u^0x_2^2)q^2=0,\;\; &u^0x_3^3 - \hat{x}_3^3 = 0\\
	(u^0x_3^3-1)p^3+u^1p^2 - u^2p^1=0,\quad (u^0x_3^3-\hat{x}_2^2)q^3 +u^1q^2 - u^2q^1=0.&
\end{split} \nonumber \end{equation}
 If $p^1q^2\neq 0$, we get with the matrix $M_{\alpha}=\begin{bmatrix}
 	1 & 0 & 0 & 0 \\
 	0 &	p^1 & 0 & 0 \\
 	0 &	0 & q^2 & 0 \\
 	0 &	0 & 0 & p^1q^2
 \end{bmatrix}$ a direct isomorphism giving $\hat{x}_2^2=x_2^2$, $\hat{x}_3^3=x_3^3$.\\
 If $p^1q^2= 0$, then $p^2q^1\neq 0$ and the matrix $M_{\alpha}=\begin{bmatrix}
	\frac{1}{x_2^2}& 0 & 0 & 0 \\
	0 &	0 & q^1 & 0 \\
	0 &	p^2 & 0 & 0 \\
	0 &	0 & 0 & -p^2q^1
\end{bmatrix}$ determines an opposite isomorphism between the algebras with $L_{e_0} = \begin{bmatrix}
	1 & 0 & 0 \\
	0 & x_{2}^2 & 0\\
	0 & 0 & x_{3}^3
\end{bmatrix}$ and $L_{e_0} = \begin{bmatrix}
	1 & 0 & 0 \\
	0 & \frac{1}{x_2^2} & 0\\
	0 & 0 & \frac{x_3^3}{x_2^2}
\end{bmatrix}$.
Hence the assertion is true.
\end{proof}
\begin{prop}\label{30diagonal}
	The $\mathcal{BL}_4$-algebras with $L_{e_0} = \begin{bmatrix}
		1 & 0 & 0 \\
		0 & x_{2}^2 & 1\\
		0 & 0 & 0
	\end{bmatrix}$ and $L_{e_0} = \begin{bmatrix}
		1 & 0 & 0 \\
		0 & \hat{x}_{2}^2 & 1\\
		0 & 0 & 0
	\end{bmatrix}$ are isomorphic if and only if they are directly isomorphic and $\hat{x}_2^2=x_2^2$.
\end{prop}
\begin{proof}
Considering the matrix equation
	\[\begin{split}\begin{bmatrix}
			p^1 & q^1\hat{x}_2^2 &  q^1\\
			p^2 & q^2\hat{x}_2^2 &  q^2\\
			p^3 & q^3\hat{x}_2^2 & q^3
		\end{bmatrix} =\begin{bmatrix}	
			u^0p^1 & u^0q^1 & 0\\
			u^0(x_2^2p^2+p_3) & u^0(x_2^2q^2+q_3) & u^0(p^1q^2-p^2q^1)\\
			u^1p^2 - u^2p^1 & u^1q^2 - u^2q^1 & 0
		\end{bmatrix}\end{split}\]
	we obtain $q^1=q^3=0$. Hence one has $p^1q^2\neq 0$.  Then $u^0 = 1$ and $\hat{x}_2^2= x_2^2$, proving the assertion.
\end{proof}

\section{Classification}

We will denote the $\mathcal{BL}_4$-algebras defined by the non-vanishing multiplication relations with respect to the canonical basis $\{e_0,e_1,e_2,e_3\}$ of $\mathbb K^4$ as follows:
 \begin{equation}\label{canonical}\begin{split}
	\mathfrak{a}_0:\quad &e_1 e_2=e_3, e_0e_3=e_3, \\
    \mathfrak{a}_1:\quad &e_1 e_2=e_3,  e_0e_1=e_1,\;  e_0e_2=- e_2,\\
	\mathfrak{b}(\lambda):\quad &e_1 e_2=e_3, e_0e_1=e_1, e_0e_2=\lambda e_2, e_0e_3=e_2,\\
	\mathfrak{c}(\lambda):\quad &e_1 e_2=e_3, e_0e_1=e_1, e_0e_2= e_2,  e_0e_3=\lambda e_3,\\
	\mathfrak{d}(\lambda,\mu):\quad &e_1 e_2=e_3, e_0e_1=e_1, e_0e_2=\mu e_2, e_0e_3=\lambda e_3, \; \mu\neq 1,\; \lambda^2+(\mu+1)^2\neq 0,
\end{split}	\end{equation}
where $\lambda, \mu \in \mathbb K$. If we omit the conditions $\mu\neq 1$, $\lambda^2+(\mu+1)^2\neq 0$, we can write $\mathfrak{a}_1=\mathfrak{d}(0,-1)$ and $\mathfrak{c}(\lambda)=\mathfrak{d}(\lambda,1)$.\\
Summarizing our previous considerations we have
\begin{prop}\label{algebralist}
 The direct isomorphism classes of $\mathcal{BL}_4$-algebras over a field $\mathbb K$ are bijectively represented by the algebras \eqref{canonical}.	
\end{prop}
There are opposite isomorphisms between some directly non-isomorphic algebras $\mathfrak{d}(\lambda,\mu)$. Taking into account Proposition \ref{eqdiagonal} (b) we get the following classification:
\begin{theor}\label{izoalgebralist} The isomorphism classes of $\mathcal{BL}_4$-algebras over $\mathbb K$ are represented bijectively by the algebras $\mathfrak{a}_0$, $\mathfrak{a}_1$, $\mathfrak{b}(\lambda)$, $\mathfrak{c}(\lambda)$, $\lambda\in\mathbb K$, and by the pairs
	$$\{\mathfrak{d}(\lambda,\mu),\mathfrak{d}(\frac{\lambda}{\mu},\frac{1}{\mu})\}\quad \lambda,\mu\in\mathbb K,\; \mu\neq 1,\; \lambda^2+(\mu+1)^2\neq 0,$$
of opposite isomorphic algebras.
\end{theor}
\begin{prop} \label{anticommutativealgebras}
The algebras $\mathfrak{a}_0$, $\mathfrak{a}_1$, $\mathfrak{b}(\lambda)$, $\mathfrak{c}(\lambda)$, $\mathfrak{d}(\lambda,\mu)$ have the following properties:
	 \begin{equation*}\begin{split}
	\mathfrak{a}_0:\quad &\text{non-Malcev binary Lie algebra}, \\
	\mathfrak{a}_1:\quad &\text{Lie algebra}, \\
	\mathfrak{b}(\lambda):\quad &\text{non-binary Lie algebra},\\
	\mathfrak{c}(\lambda):\quad &\text{Lie algebra},\hspace{100pt} \text{if}\quad\lambda=2,\\
	     &\text{non-Lie Malcev algebra},\hspace{42pt} \text{if}\quad\lambda=-1,\\
		 &\text{non-Malcev binary Lie algebra},\hspace{9pt} \text{if}\quad\lambda \neq \{-1,2\}  \\
    \{\mathfrak{d}(\lambda,\mu),\mathfrak{d}(\frac{\lambda}{\mu},\frac{1}{\mu})\}:\quad &\text{pair of Lie algebras},\hspace{63pt} \text{if}\quad \lambda=\mu + 1,\\
    & \text{pair of non-binary Lie algebras,}\hspace{10pt} \text{if}\quad \lambda\neq\mu + 1.
\end{split}\end{equation*}
\end{prop}
\begin{proof}
The assertion with respect to $\mathfrak{a}_0$ and $\mathfrak{c}(\lambda)$ are contained in Gainov's theorem (cf. \eqref{gainov}). Putting $x=e_0$, $y=e_1+e_3$ into the identity \eqref{binarylieiden} we get
\[\begin{split}((e_1+e_3)\cdot e_0(e_1+e_3))e_0+(e_0(e_1+e_3)\cdot e_0)(e_1+e_3)
	= -e_2+\lambda e_3 \neq 0, \end{split}\]
which means that the algebras $\mathfrak{b}(\lambda)$ are not binary Lie. For an algebra $\mathfrak{d}(\lambda,\mu)$ an easy computation shows that $\mathcal{J}(\mathfrak{d}(\lambda,\mu))= \mathbb K(\lambda - \mu -1)e_3$, hence it is a Lie algebra if and only if $\lambda=\mu + 1$. Otherwise the algebra is not binary Lie, since  all $4$-dimensional binary Lie algebras are among $\mathfrak{a}_0$ and $\mathfrak{c}(\lambda)$.
\end{proof}

\section{Groups which are candidates for automorphism groups}

In the following we consider block matrices of the form
\begin{equation}\label{fijeloles}\begin{split} &\mathcal{A}=\begin{bmatrix}
			A & 0 \\
			{\bf a}^t & |A|
		\end{bmatrix},\; A\in \text{GL}(2,\mathbb K),\; {\bf a}\in \mathbb K^2,\\
	&\begin{bmatrix}
	1 & 0 & 0 \\
	\phi(A,{\bf a}) &	A & 0 \\
	\alpha & {\bf a}^t & |A|
\end{bmatrix},\; A\in \text{GL}(2,\mathbb K),\; {\bf a}\in \mathbb K^2,\; \alpha\in\mathbb K,\; \phi: \text{GL}(2,\mathbb K)\times\mathbb K^2\to\mathbb K^2,\end{split} \nonumber \end{equation}
where the $3\times 3$ matrix $\mathcal{A}$ determines an automorphism of the ideal $\mathfrak{p}$ isomorphic to $\mathfrak{n}$  of the algebra $\mathfrak{c}$. According to the multiplication
$$\begin{bmatrix}
	1 & 0 & 0 \\
	\phi(A,{\bf a}) & A  & 0 \\
	\alpha &{\bf a}^t & |A|
\end{bmatrix}\cdot \begin{bmatrix}
	1 & 0 & 0 \\
	\phi(B,{\bf b}) &  B & 0 \\
	\beta &{\bf b}^t & |B|
\end{bmatrix} = \begin{bmatrix}
	1 & 0 & 0 \\
	\phi(A,{\bf a}) + A\phi(B,{\bf b}) &	AB & 0 \\
	\alpha + {\bf a}^t\phi(B,{\bf b}) + |A|\beta & {\bf a}^t B + |A|{\bf b}^t & |AB|
\end{bmatrix},$$
the set
$$\left\{\begin{bmatrix}
	1 & 0 & 0 \\
	\phi(A,{\bf a}) &	A & 0 \\
	\alpha & {\bf a}^t & |A|
\end{bmatrix},\; A\in \text{GL}(2,\mathbb K),\; {\bf a}\in \mathbb K^2,\; \alpha\in\mathbb K,\; \phi: \text{GL}(2,\mathbb K)\times\mathbb K^2\to\mathbb K^2\right\}$$
is a group, if and only if the map $\phi:\text{GL}(2,\mathbb K)\times\mathbb K^2\to\mathbb K^2$ satisfies the functional equation
\begin{equation}\label{funeq}
	\phi(A,{\bf a}) + A\phi(B,{\bf b}) = \phi (AB,({\bf a}^tB+ |A|{\bf b}^t)^t).
\end{equation}
If $\phi$ fulfills \eqref{funeq} then the obtained group is denoted by $\mathcal{G}_\phi(\mathbb K,\Aut(\mathfrak{n}))$.
Hence we have
\begin{lem} \label{csopbov}
The group $\mathcal{G}_\phi(\mathbb K,\Aut(\mathfrak{n}))$ is isomorphic to the group extension of $\Aut(\mathfrak{n})$ by the normal subgroup $\mathbb K$ defined on $\Aut(\mathfrak{n})\times \mathbb K$ by the  multiplication
\begin{equation}\label{groupext}(\mathcal{A},\alpha)\cdot(\mathcal{B},\beta)= (\mathcal{A}\mathcal{B},\alpha + {\bf a}^t\phi(B,{\bf b}) + |A|\beta).\end{equation}
\end{lem}
Note that the group extension given by (\ref{groupext}) is a semidirect product if and only if
${\bf a}^t\phi(B,{\bf b})$ is identically zero.

In the following we will identify the matrix group $\mathcal{G}_{\tau\phi}(\mathbb K,\Aut(\mathfrak{n}))$ with the corresponding group extension.
Assume that the map $\phi$ satisfies \eqref{funeq}, then for any $0\neq \tau\in\mathbb K$ also the map $\bar{\phi}_\tau:=\tau\phi$ satisfies \eqref{funeq}. Comparing the multiplication
$$(\mathcal{A},\alpha)\star(\mathcal{B},\beta)= (\mathcal{A}\mathcal{B},\alpha + {\bf a}^t\bar{\phi}_\tau(B,{\bf b}) + |A|\beta)= (\mathcal{A}\mathcal{B},\alpha + \tau{\bf a}^t\phi(B,{\bf b}) + |A|\beta)$$
of the group $\mathcal{G}_{\tau\phi}(\mathbb K,\Aut(\mathfrak{n}))$ with the multiplication \eqref{groupext} we obtain
\begin{lem} \label{lemiso} The map $\Psi:\mathcal{G}_\phi(\mathbb K,\Aut(\mathfrak{n})) \to \mathcal{G}_{\tau\phi}(\mathbb K,\Aut(\mathfrak{n}))$ defined by  $\Psi:(\mathcal{A},\alpha)\mapsto(\mathcal{A},\tau\alpha)$ is an isomorphism.
\end{lem}
We notice that in matrix notation for the groups $\mathcal{G}_\phi(\mathbb K,\Aut(\mathfrak{n}))$ and $\mathcal{G}_{\tau\phi}(\mathbb K,\Aut(\mathfrak{n}))$ the isomorphism $\Psi$ has the form
$$\begin{bmatrix}
	1 & 0 & 0 \\
	\phi(A,{\bf a}) &	A & 0 \\
	\alpha & {\bf a}^t & |A|
\end{bmatrix}\mapsto\begin{bmatrix}
1 & 0 & 0 \\
\tau\phi(A,{\bf a}) &	A & 0 \\
\tau\alpha & {\bf a}^t & |A|
\end{bmatrix}.$$
Assume that the
$2\times 2$-matrix $A$, the vectors ${\bf a}$, ${\bf 0}$, and the map $\phi: \text{GL}(2,\mathbb K)\times\mathbb K^2\to\mathbb K^2$ are given by
\begin{equation}\label{fifinulla} A =  \begin{bmatrix}
		p^1 & q^1 \\
		p^2 & q^2
	\end{bmatrix}, \; {\bf a} =  \begin{bmatrix}
		p^3 \\
		q^3
\end{bmatrix},\; {\bf 0} =  \begin{bmatrix}
	0\\
	0
\end{bmatrix},\; \phi(A,{\bf a}) = \begin{bmatrix}
		\frac{p^3 q^1- p^1q^3}{p^1q^2-p^2q^1} \\
		\frac{p^3 q^2- p^2q^3}{p^1q^2-p^2q^1} \\
	\end{bmatrix},\; p^i,q^i\in\mathbb K,\; i=1,2,3. \end{equation}
\begin{prop}
 The sets
	\begin{equation}\label{gammagroups} \begin{split} &\Gamma =\left\{ \begin{bmatrix}
		1 &{\bf 0}^t & 0 \\
		\phi(A,{\bf a}) & A &  {\bf 0} \\
		u & {\bf a}^t & |A|
	\end{bmatrix}; u\in\mathbb{K},\; A\in \text{GL}(2,\mathbb K),\; {\bf a}\in \mathbb K^2 \right\},\\
&\Gamma_0 =\left\{ \begin{bmatrix}
	1 &{\bf 0}^t & 0 \\
	{\bf 0} & A &  {\bf 0} \\
	u & {\bf a}^t & |A|
\end{bmatrix}; u\in\mathbb{K},\; A\in \text{GL}(2,\mathbb K),\; {\bf a}\in \mathbb K^2\right\},\\
&\Gamma_\triangle =\left\{ \begin{bmatrix}
	1 &{\bf 0}^t & 0 \\
	{\bf 0} & A &  {\bf 0} \\
	u & {\bf 0}^t & |A|
\end{bmatrix}\in\Gamma;\; \text{with}\; A =  \begin{bmatrix}
	1 & 0 \\
	p^2 & q^2
\end{bmatrix}\right\}\subset \Gamma,\\
&\Gamma^{(+)} =\left\{ \begin{bmatrix}
	1 &{\bf 0}^t & 0 \\
	\phi(A,{\bf a}) & A &  {\bf 0} \\
	u & {\bf a}^t & |A|
\end{bmatrix}\in\Gamma;\; \text{with}\; A =  \begin{bmatrix}
	p^1 & 0 \\
	0 & q^2
\end{bmatrix}\right\}\subset \Gamma,\\
&\Gamma^{(-)} =\left\{ \begin{bmatrix}
	1 &{\bf 0}^t & 0 \\
	\Theta \phi(A,{\bf a}) & A &  {\bf 0} \\
	u & {\bf a}^t & |A|
\end{bmatrix};\; \text{with}\; A =  \begin{bmatrix}
	p^1 & 0 \\
	0 & q^2
\end{bmatrix}\; \text{and}\;\Theta = \begin{bmatrix}
-1 & 0 \\
0 & 1
\end{bmatrix}\right\}\end{split} \nonumber \end{equation}
are pairwise non-isomorphic groups, which are group extensions of a subgroup of $$\Aut(\mathfrak{n})=\left\{ \begin{bmatrix}
	A &  {\bf 0} \\
	{\bf a}^t & |A|
\end{bmatrix}, A\in \text{GL}(2,\mathbb K),\; {\bf a}\in \mathbb K^2 \right\}$$ by the normal subgroup $\{u\in\mathbb K\}$, particularly, $\Gamma_0$ and $\Gamma_\triangle$ are semidirect products.
\end{prop}
\begin{proof}
An easy computation shows that the maps $\phi: \text{GL}(2,\mathbb K)\times\mathbb K^2\to\mathbb K^2$ and $\Theta\phi$, defined on the subset of $\text{GL}(2,\mathbb K)\times\mathbb K^2$ determined by $p^2=q^1=0$, satisfy the functional equation \eqref{funeq}, hence $\Gamma$, $\Gamma_0$, $\Gamma_\triangle$, $\Gamma^{(+)}$, $\Gamma^{(-)}$ are groups. The elements of the groups $\Gamma$ and $\Gamma_0$ depend on $7$ parameters, the elements of $\Gamma_\triangle$ on $3$ parameters and of $\Gamma^{(+)}$, $\Gamma^{(-)}$ on $5$ parameters.  According to Lemma \ref{csopbov} all the groups  $\Gamma$, $\Gamma_0$, $\Gamma_\triangle$, $\Gamma^{(+)}$, $\Gamma^{(-)}$ are group extensions of a subgroup of $\Aut(\mathfrak{n})$ by the normal subgroup $\{u\in\mathbb K\}$, particularly, $\Gamma_0$ and $\Gamma_\triangle$ are semidirect products.

The  groups $\Gamma^{(+)}$, $\Gamma^{(-)}$ are solvable groups. Let $E$ be the $2 \times 2$ identity matrix.  The kernels $$T^{(+)}= \left\{ \begin{bmatrix}
	1 &{\bf 0}^t & 0 \\
	\phi(E,{\bf a}) & E & {\bf 0} \\
	u & {\bf a}^t & 1
\end{bmatrix}\right\}\quad\text{and}\quad T^{(-)}= \left\{ \begin{bmatrix}
1 &{\bf 0}^t & 0 \\
\Theta\phi(E,{\bf a}) & E & {\bf 0} \\
u & {\bf a}^t & 1
\end{bmatrix}\right\},$$
respectively, of the homomorphism
onto the subgroup $U=\left\{ \begin{bmatrix}
	1 &{\bf 0}^t & 0 \\
	{\bf 0} & A & {\bf 0} \\
	0 & {\bf 0}^t & |A|
\end{bmatrix}\right\}$ of diagonal matrices are normal unipotent subgroups. We obtain the semidirect products $\Gamma^{(+)}=T^{(+)}\ltimes U$ and $\Gamma^{(-)}=T^{(-)}\ltimes U$, since any element can be written in the form $ut = t\cdot t^{-1}ut$, $u\in U$, $t\in T^{(+)}$ or $t\in T^{(-)}$, respectively,  where the kernel subgroup is acting on $U$ by conjugation $u\mapsto t^{-1}ut$. Applying Lemma \ref{csopbov} to the kernel subgroups, we obtain that the groups $T^{(+)}$ and $T^{(-)}$ are isomorphic to the group extensions of $\mathbb K^2$ by the normal subgroup $\mathbb K$ having the  multiplication
$$({\bf a},u)\cdot({\bf b},v)= ({\bf a + b},u + v + {\bf a}^t\phi(E,{\bf b})),\quad ({\bf a},u)\cdot({\bf b},v)= ({\bf a + b},u + v + {\bf a}^t\Theta\phi(E,{\bf b})),$$
respectively. If ${\bf a}^t = [p^3,q^3]$, ${\bf b}^t = [\hat{p}^3,\hat{q}^3]$ we get using \eqref{fifinulla} that
$${\bf a}^t\phi(E,{\bf b}) = [p^3,q^3]\begin{bmatrix}
	-\hat{q}^3\\
    \hat{p}^3\\
\end{bmatrix}= q^3\hat{p}^3-p^3\hat{q}^3,\quad {\bf a}^t\Theta\phi(E,{\bf b}) = [p^3,q^3]\begin{bmatrix}
\hat{q}^3\\
\hat{p}^3\\
\end{bmatrix}= q^3\hat{p}^3+p^3\hat{q}^3$$
It follows that the subgroup $T^{(+)}$ is non-commutative, but the subgroup  $T^{(-)}$ is commutative. Hence the groups $\Gamma^{(+)}$ and $\Gamma^{(-)}$ are non-isomorphic.
\end{proof}

\section{Classification of groups of automorphisms}\label{osztalyozas}

\begin{prop}\label{a331} The group of automorphisms of the algebra $\mathfrak{a}_0$ as well as of the algebras $\mathfrak{c}(\lambda)$, $1\neq\lambda\in\mathbb K$, is the group	
	$\Aut(\mathfrak{a}_0) =\Aut(\mathfrak{c}(\lambda)) =\Gamma$.\\ 
	Moreover, the group of automorphisms of the algebra $\mathfrak{c}(1)$ is the group	
	$\Aut(\mathfrak{c}(1)) = \Gamma_0$.  
\end{prop}
\begin{proof}
	In case of the algebra $\mathfrak{a}_0$ from the matrix equation 
	\[\begin{bmatrix}
		0 & 0 & 0 \\
		0 & 0 & 0 \\
		0 & 0& p^1 q^2-p^2 q^1
	\end{bmatrix} =\begin{bmatrix}	
		0 & 0 & 0\\
		0 & 0 & 0\\
		u^0 p^3+u^1p^2 - u^2p^1 & u^0 q^3 +u^1q^2 - u^2q^1 & u^0(p^1 q^2-p^2 q^1)
	\end{bmatrix} \]
	it follows that $u^0=1$ because of $p^1 q^2-p^2 q^1 \neq 0$. Hence we obtain that $p^3=u^2 p^1- u^1 p^2$ and 
	$q^3=u^2 q^1- u^1 q^2$ and equivalently  
	we get $u^1=\frac{p^3q^1-q^3p^1}{p^1q^2-p^2q^1}$, $u^2=\frac{p^3q^2-q^3p^2}{p^1q^2-p^2q^1}$. This proves the assertion for the algebra $\mathfrak{a}_0$. 
	
	The algebra $\mathfrak{c}(\lambda)$, $\lambda\in\mathbb K$, arises from the algebra $\mathfrak{d}(\lambda,\mu)$ with $\mu=1$. 
	Since $p^1 q^2-p^2 q^1 \neq 0$ from the matrix equation 
	\[\begin{split}& \begin{bmatrix}
			p^1 & q^1 & 0\\
			p^2 & q^2 & 0\\
			p^3 & q^3 & (p^1q^2-p^2q^1)\lambda
		\end{bmatrix}  =\\=& \begin{bmatrix}	
			u^0 p^1 & u^0 q^1 & 0\\
			u^0 p^2 & u^0 q^2 & 0\\
			u^0 \lambda p^3+u^1p^2 - u^2p^1 & u^0 \lambda q^3 +u^1q^2 - u^2q^1 & u^0 \lambda(p^1q^2-p^2q^1)
	\end{bmatrix} \end{split} \] 
	it follows that $u^0=1$. Furthermore, we obtain $p^3(\lambda -1)=u^2p^1- u^1p^2$, $q^3(\lambda - 1)=u^2q^1 - u^1q^2$ and equivalently $u^1=\frac{p^3q^1- p^1q^3}{p^1q^2-p^2q^1}(\lambda -1)$ and $u^2=\frac{p^3 q^2- p^2q^3}{p^1q^2-p^2q^1}(\lambda -1)$.  Therefore we receive  
	\begin{equation}\label{Amatrixaut6} \Aut(\mathfrak{c}(\lambda)) =\left\{ \begin{bmatrix}
			1 &{\bf 0}^t & 0 \\
			\phi(A,{\bf a})(\lambda -1) & A & {\bf 0} \\
			u^3 & {\bf a}^t & |A| 
		\end{bmatrix}, u^3\in\mathbb{K},\; A\in \text{GL}(2,\mathbb K),\; {\bf a}\in \mathbb K^2 \right\}.\end{equation} 
	Applying Lemma \ref{lemiso} with $\tau =(\lambda -1)^{-1}$ to the  group of automorphisms given by (\ref{Amatrixaut6}) we obtain the assertion for the algebras $\mathfrak{c}(\lambda)$, $1\neq\lambda\in\mathbb K$. Putting $\lambda = 1$ in (\ref{Amatrixaut6}) we have the assertion for the automorphism group of the algebra $\mathfrak{c}(1)$.  	
\end{proof}  
\begin{prop}\label{0felsotrian} Let be 
	$B(p^2, q^2 ) =  \begin{bmatrix}
		1 & 0 \\
		p^2 & q^2
	\end{bmatrix}$. 
	The group $\Aut(\mathfrak{b}(\lambda))$ of automorphisms of the algebras $\mathfrak{b}(\lambda)$, $\lambda\in\mathbb K$, 
	is	the group $\Aut(\mathfrak{b}(\lambda)) = \Gamma_\triangle$. 
\end{prop}
\begin{proof}
	From the matrix equation 
	\[\begin{bmatrix}
		p^1 & q^1\lambda & q^1 \\
		p^2 & q^2\lambda & q^2 \\
		p^3 & q^3\lambda & q^3
	\end{bmatrix} =\begin{bmatrix}	
		u^0p^1 & u^0q^1 & 0\\
		u^0(\lambda p^2+p^3) & u^0(\lambda q^2+q^3) & u^0(p^1 q^2-p^2 q^1)\\
		u^1p^2 - u^2p^1 &  u^1q^2 - u^2q^1 & 0
	\end{bmatrix} \]
	it follows that $q^1=q^3=0$ and hence $p^1 q^2 \neq 0$ and $u^0=1$. Since $q^2=p^1 q^2$ and $q^2 \neq 0$ we have $p^1=1$ and $u^1=0$. Furthermore we obtain $p^3=p^2(1-\lambda)$ and $u^2=-p^3=p^2(\lambda-1)$. 
	The group $\Aut(\mathfrak{b}(\lambda))$ of automorphisms of the algebras $\mathfrak{b}(\lambda)$ consists of matrices 
	\begin{equation}\label{Amatrixaut1lambda} \begin{bmatrix}
			1 &0 & 0 & 0 \\
			0 &1 & 0 & 0 \\
			p^2(\lambda-1) &p^2 & q^2 & 0 \\
			u^3 &p^2(1-\lambda) & 0 & q^2 
		\end{bmatrix},\ u^3,\  p^2, q^2 \in\mathbb{K}.\end{equation} 
	with multiplication  
	$$ \begin{bmatrix}
		1 &0 & 0 & 0 \\
		0 &1 & 0 & 0 \\
		p^2(\lambda-1) &p^2 & q^2 & 0 \\
		u^3 &p^2(1-\lambda) & 0 & q^2 
	\end{bmatrix}\begin{bmatrix}
		1 &0 & 0 & 0 \\
		0 &1 & 0 & 0 \\
		r^2(\lambda-1) &r^2 & s^2 & 0 \\
		v^3 &r^2(1-\lambda) & 0 & s^2 
	\end{bmatrix} =$$
	$$= \begin{bmatrix}
		1 &0 & 0 & 0 \\
		0 &1 & 0 & 0 \\
		(p^2+q^2r^2)(\lambda-1) &p^2+q^2r^2 & q^2s^2 & 0 \\
		u^3+q^2v^3 &(p^2+q^2r^2)(1-\lambda) & 0 & q^2s^2 
	\end{bmatrix}. $$
	Since $$ \begin{bmatrix}
		1 &0 & 0 & 0 \\
		0 &1 & 0 & 0 \\
		0 &p^2 & q^2 & 0 \\
		u^3 & 0 & 0 & q^2 
	\end{bmatrix}\begin{bmatrix}
		1 &0 & 0 & 0 \\
		0 &1 & 0 & 0 \\
		0 &r^2 & s^2 & 0 \\
		v^3 & 0 & 0 & s^2 
	\end{bmatrix} = \begin{bmatrix}
		1 &0 & 0 & 0 \\
		0 &1 & 0 & 0 \\
		0 &p^2+q^2r^2 & q^2s^2 & 0 \\
		u^3+q^2v^3 & 0 & 0 & q^2s^2 
	\end{bmatrix},$$
	we see that the map 
	$$\begin{bmatrix}
		1 &0 & 0 & 0 \\
		0 &1 & 0 & 0 \\
		p^2(\lambda-1) &p^2 & q^2 & 0 \\
		u^3 &p^2(1-\lambda) & 0 & q^2 
	\end{bmatrix}\mapsto \begin{bmatrix}
		1 &0 & 0 & 0 \\
		0 &1 & 0 & 0 \\
		0 &p^2 & q^2 & 0 \\
		u^3 & 0 & 0 & q^2 
	\end{bmatrix}$$
	is an isomorphism between (\ref{Amatrixaut1lambda}) and $\Gamma_\triangle$, proving the assertion.
\end{proof} 
Let us denote 
\begin{equation}\label{fifinullauj}\begin{split} &D =  \begin{bmatrix}
			p^1 & 0 \\
			0 & q^2
		\end{bmatrix}, \quad {\bf a} =  \begin{bmatrix}
			p^3 \\
			q^3
		\end{bmatrix},\quad {\bf N} =  \begin{bmatrix}
			\alpha & 0\\
			0 & \beta
		\end{bmatrix},\quad \phi(D,{\bf a}) = \begin{bmatrix}
			-\frac{q^3}{q^2} \\
			\frac{p^3}{p^1} \\
		\end{bmatrix},\\ 
		&\Gamma(\alpha,\beta) =\left\{ \begin{bmatrix}
			1 & {\bf 0}^t & 0  \\
			{\bf N}\phi(D,{\bf a}) & D & {\bf 0}  \\
			u^3 & {\bf a}^t & |D| 
		\end{bmatrix},\ u^3,\  p^1, q^2, p^3, q^3 \in\mathbb{K}\right\}, \quad\alpha,\beta\in\mathbb K.\end{split} \nonumber \end{equation}
\begin{prop}\label{Amatrixaut5} Assume $\mu\neq 1$.
	The groups $\mathcal{\Aut}(\mathfrak{d}(\lambda,\mu))$ of automorphisms of the algebras $\mathfrak{d}(\lambda,\mu)$, 
	\begin{enumerate}
		\item[(A)] contain a subgroup isomorphic to 
		\begin{equation}\label{autgammapluszminusz} \Gamma(\pm 1,1)=\begin{cases} \Gamma(1,1)\quad\;\;\text{if}\quad\mu - 2\lambda + 1 \neq 0,\\ \Gamma(-1,1)\quad\text{if}\quad\mu - 2\lambda + 1 = 0.\end{cases} \nonumber \end{equation}
		\item[(B)]  	If $\lambda^2+(\mu+1)^2\neq 0$ then $\mathcal{\Aut}(\mathfrak{d}(\lambda,\mu))\cong\Gamma(\pm 1,1)$.
	\end{enumerate}
\end{prop}
\begin{proof}
	The first and second rows and the third column of the matrix equation 
	\begin{equation}\label{Amatrixadlambdamu}\begin{split}&\begin{bmatrix}
				p^1 & q^1\mu & 0\\
				p^2 & q^2\mu & 0\\
				p^3 & q^3\mu & (p^1q^2-p^2q^1)\lambda
			\end{bmatrix} =\\=&\begin{bmatrix}	
				u^0p^1 & u^0q^1 & 0\\
				u^0\mu p^2 & u^0\mu q^2 & 0\\
				u^0\lambda p^3+u^1p^2 - u^2p^1 &  u^0\lambda q^3 +u^1q^2 - u^2q^1 & u^0\lambda(p^1q^2-p^2q^1)
	\end{bmatrix}\end{split} \end{equation}	 
	give the equations 
	\begin{equation}\label{firstsecondrows} p^1(u^0-1)=0,\quad q^1(u^0-\mu) = 0, \quad p^2(u^0\mu-1)=0,\quad q^2\mu(u^0-1) = 0,\quad \lambda(u^0-1)=0.\end{equation}
	Assume that $u^0 = 1$. Since $\mu\neq 1$, we get from \eqref{firstsecondrows} that $p^2 = q^1 = 0$. The third row of \eqref{Amatrixadlambdamu} gives  
	$p^3(\lambda-1)=u^2 p^1$ and $q^3(\mu-\lambda)=u^1 q^2$, hence 
	\begin{equation}\label{unullegy}u^1=-\frac{q^3}{q^2}(\lambda-\mu),\quad  u^2=\frac{p^3}{p^1}(\lambda-1). \end{equation}
	Consequently, the solutions of the equation \eqref{Amatrixadlambdamu} satisfying $u^0=1$ determine the subgroup	
	\begin{equation}\label{subgrouplambdamu} \Gamma(\lambda-\mu,\lambda-1)\subseteq\mathcal{\Aut}(\mathfrak{d}(\lambda,\mu)).\end{equation}	 
	\begin{lem}\label{gammapluszminusz}
		The subgroup  $\Gamma(\lambda-\mu,\lambda-1)\subseteq \mathcal{\Aut}(\mathfrak{d}(\lambda,\mu))$ is isomorphic to 
		\begin{enumerate}
			\item[(a)] $\Gamma(-1,1)$, if $\mu - 2\lambda + 1 =0$,
			\item[(b)] $\Gamma(1,1)$, if $\mu - 2\lambda + 1 \neq 0$.
		\end{enumerate}
	\end{lem}
	\begin{proof}		
		If $\mu - 2\lambda + 1 =0$, then $\Gamma(\lambda-\mu,\lambda-1)=\Gamma(-(\lambda-1),\lambda-1)$. Lemma \ref{lemiso} implies  $$
		\Gamma(-1,1)\cong \Gamma(-(\lambda-1),\lambda-1)\subset \mathcal{\Aut}(\mathfrak{d}(\lambda,\mu)),$$
		proving assertion (a).\\
		Assume $\mu - 2\lambda + 1 \neq 0$.
		The group $\Gamma(\lambda-\mu,\lambda-1)$ is a solvable algebraic group for any $\lambda-\mu, \lambda-1 \in \mathbb K$. 
		The map
		$$\begin{bmatrix}
			1 &0 & 0 & 0 \\
			-q^3(\lambda-\mu) & p^1 & 0 & 0 \\
			p^3(\lambda-1) &0 & q^2 & 0 \\
			u^3 &p^3 & q^3 & p^1q^2 
		\end{bmatrix}\mapsto \begin{bmatrix}
			1 &0 & 0 & 0 \\
			0 & p^1 & 0 & 0 \\
			0 &0 & q^2 & 0 \\
			0 & 0 & 0 & p^1q^2 
		\end{bmatrix}$$ gives a homomorphism $\Gamma(\lambda-\mu,\lambda-1)\to T$, where $T$ is the $\mathbb K$-torus  consisting of diagonal matrices.  The kernel $U(\lambda-\mu,\lambda-1) = \left\{ \begin{bmatrix}
			1 &0 & 0 & 0 \\
			-q^3(\lambda-\mu) & 1 & 0 & 0 \\
			p^3(\lambda-1) &0 & 1 & 0 \\
			u^3 &p^3 & q^3 & 1 
		\end{bmatrix},\ u^3,\ p^3, q^3 \in\mathbb{K}\right\}$ of this homomorphism 
		is a normal unipotent subgroup. The group $\Gamma(\lambda-\mu,\lambda-1)$ is a semidirect product $T\ltimes U(\lambda-\mu,\lambda-1)$, since any element can be written into the form $ut = t\cdot t^{-1}ut$, $u\in U(\lambda-\mu,\lambda-1)$, $t\in T$, where $T$  is acting on $U(\lambda-\mu,\lambda-1)$ by conjugation. The Baker--Campbell--Hausdorff formula defines a group $\mathcal{G}(\mathfrak{u}(\lambda-\mu,\lambda-1))$ on the nilpotent Lie algebra $$\mathfrak{u}(\lambda-\mu,\lambda-1) = \left\{ \begin{bmatrix}
			0 &0 & 0 & 0 \\
			-y(\lambda-\mu) & 0 & 0 & 0 \\
			x(\lambda-1) &0 & 0 & 0 \\
			z &x & y & 0 
		\end{bmatrix},\ x,y,z \in\mathbb{K}\right\}$$ of $U(\lambda-\mu,\lambda-1)$, hence the group  $\mathcal{G}(\mathfrak{u}(\lambda-\mu,\lambda-1))$ is isomorphic to $U(\lambda-\mu,\lambda-1)$. Considering the action of $T$ on $\mathcal{G}(\mathfrak{u}(\lambda-\mu,\lambda-1))$ by conjugation induced by the action of $T$ on $U(\lambda-\mu,\lambda-1)$ we get, that the group 
		$\Gamma(\lambda-\mu,\lambda-1)$ is isomorphic to the semidirect product $T\ltimes \mathcal{G}(\mathfrak{u}(\lambda-\mu,\lambda-1))$.  The Lie bracket operation on $\mathfrak{u}(\lambda-\mu,\lambda-1)$ is given by 
		\[\begin{split}
			&\left[ \begin{bmatrix}
				0 &0 & 0 & 0 \\
				-y_1(\lambda-\mu) & 0 & 0 & 0 \\
				x_1(\lambda-1) &0 & 0 & 0 \\
				z_1 &x_1 & y_1 & 0 
			\end{bmatrix},\begin{bmatrix}
				0 &0 & 0 & 0 \\
				-y_2(\lambda-\mu) & 0 & 0 & 0 \\
				x_2(\lambda-1) &0 & 0 & 0 \\
				z_2 &x_2 & y_2 & 0 
			\end{bmatrix} \right] =\\& = \begin{bmatrix}
				0 &0 & 0 & 0 \\
				0 & 0 & 0 & 0 \\
				0 &0 & 0 & 0 \\
				(\mu-2\lambda+1)(x_1y_2-x_2y_1) & 0 & 0 & 0 
			\end{bmatrix},
		\end{split}\]
		hence the map $ \begin{bmatrix}
			0 &0 & 0 & 0 \\
			-y(\lambda-\mu) & 0 & 0 & 0 \\
			x(\lambda-1) &0 & 0 & 0 \\
			z &x & y & 0 
		\end{bmatrix}\mapsto \begin{bmatrix}
			0 &0 & 0 & 0 \\
			\frac{\mu-2\lambda+1}{2}y & 0 & 0 & 0 \\
			-\frac{\mu-2\lambda+1}{2} x  &0 & 0 & 0 \\
			z &x & y & 0 
		\end{bmatrix}$ 
		determines a Lie algebra isomorphism $\mathfrak{u}(\lambda-\mu,\lambda-1)\to\mathfrak{u}(-\frac{\mu-2\lambda+1}{2},-\frac{\mu-2\lambda+1}{2})$. 
		Consequently, we get the group isomorphism $T\ltimes \mathcal{G}(\mathfrak{u}(\lambda-\mu,\lambda-1))\to T\ltimes \mathcal{G}(\mathfrak{u}(-\frac{\mu-2\lambda+1}{2},-\frac{\mu-2\lambda+1}{2}))$. Using the isomorphisms
		\[\begin{split}\Gamma(\lambda-\mu,\lambda-1)\cong T\ltimes U(\lambda-\mu,\lambda-1)\cong T\ltimes \mathcal{G}(\mathfrak{u}(\lambda-\mu,\lambda-1)),\\
		\end{split}\]
		and applying Lemma \ref{lemiso} we obtain
		$$\Gamma(\lambda-\mu,\lambda-1)\cong\Gamma(-\frac{\mu-2\lambda+1}{2},-\frac{\mu-2\lambda+1}{2})\cong\Gamma(1,1).$$Hence 
		$\Gamma(1,1)\cong\Gamma(\lambda-\mu,\lambda-1)\subseteq\mathcal{\Aut}(\mathfrak{d}(\lambda,\mu))$, 
		and we get the assertion (b).
	\end{proof} 
	By relation \eqref{subgrouplambdamu} we obtain from Lemma \ref{gammapluszminusz} the assertion of  Proposition \ref{Amatrixaut5} (A).\par
	Assume $u^0\neq 1$. From \eqref{firstsecondrows} follow $p^1 =0$, $p^2q^1\neq 0$ and $u^0-\mu = u^0\mu-1=\lambda =0$. Since $\mu\neq 1$ we get $(\lambda,\mu) = (0,-1)$. Hence if $(\lambda,\mu)\neq (0,-1)$ then necessarily $u^0=1$ and we have the isomorphism $\Gamma(\lambda-\mu,\lambda-1)\cong\mathcal{\Aut}(\mathfrak{d}(\lambda,\mu))$, proving Proposition  \ref{Amatrixaut5} (B).
\end{proof}
Consider the case $(\lambda,\mu)=(0,-1)$ we investigate the group $\Aut(\mathfrak{a}_1) =\Aut(\mathfrak{d}(0,-1))$ of automorphisms of the Lie algebra $\mathfrak{a}_1$.
\begin{prop}
	The group $\mathcal{\Aut}(\mathfrak{a}_1)$ is isomorphic to the group 
	\begin{equation}\label{Amatrixaut5uu} 
		\mathcal{\Aut}(\mathfrak{a}_1)= \Gamma(-1,1)\cup(\Gamma(-1,1)\cdot\sigma),  \end{equation}
	where $\sigma $ is the involution $\sigma = \begin{bmatrix}
		-1& 0 & 0 & 0 \\
		0 &	0 & 1 & 0 \\
		0 &	1 & 0 & 0 \\
		0 &	0 & 0 & -1
	\end{bmatrix}\in\mathcal{\Aut}(\mathfrak{a}_1).$  
\end{prop}
\begin{proof}
	Since $(\lambda,\mu)=(0,-1)$ we get $\Gamma(1,-1)\subseteq\Aut(\mathfrak{a}_1)$ from the relation \eqref{subgrouplambdamu}. Additionally, if $u^0\neq 1$ then the equations \eqref{firstsecondrows} imply $p^1=0$, $p^2q^1\neq 0$,  $u^0 = \mu$, $\mu = -1$, $q^2=0$ and the third row of equation \eqref{Amatrixadlambdamu} and the equations \eqref{unullegy} give the solutions $$u^1 = \frac{p^3}{p^2},\; u^2 = \frac{q^3}{q^1}\quad\text{and}\quad u^1=-\frac{q^3}{q^2}
	,\;  u^2=-\frac{p^3}{p^1}.$$ 
	Hence the group of automorphisms is the union $\mathcal{\Aut}(\mathfrak{a}_1)=\Gamma(1,-1)\cup\Lambda$ of the group $\Gamma(1,-1)$ and the set 
	\begin{equation}\label{mallekosztaly}\Lambda = \left\{ \begin{bmatrix}
			-1 &0 & 0 & 0 \\
			\frac{p^3}{p^2} &0 & q^1 & 0 \\
			\frac{q^3}{q^1} &p^2 & 0 & 0 \\
			u^3 &p^3 & q^3 & -p^2q^1 
		\end{bmatrix},\ u^3,\  p^2, q^1, p^3, q^3 \in\mathbb{K}\right\}. \nonumber \end{equation}		
	Denote $A^{(1)}$, $B^{(1)}$, $X^{(1)}$ diagonal $2\times 2$-matrices and $A^{(-1)}$, $B^{(-1)}$, $X^{(-1)}$ antidiagonal matrices of the form $\begin{bmatrix}
		0 & q^1\\
		p^2 & 0
	\end{bmatrix}$, $p^2, q^1\in \mathbb K$. If $X$ is a matrix of the form $\begin{bmatrix}
		p^1 & 0\\
		0 & q^2
	\end{bmatrix}$ or $\begin{bmatrix}
		0 & q^1\\
		p^2 & 0
	\end{bmatrix}$ and ${\bf x}=\begin{bmatrix}
		p^3\\
		q^3
	\end{bmatrix}$ then define the map $(X^{(\epsilon)},{\bf x})\mapsto\phi(X^{(\epsilon)},{\bf x})$, $\epsilon = \pm 1$, by 
	\begin{equation} \label{figammalambda} \phi(X^{(1)},{\bf x}) = 
		-\begin{bmatrix}
			\frac{q^3}{q^2} \\
			\frac{p^3}{p^1} \\
		\end{bmatrix},\quad \text{and}\quad \phi(X^{(2)},{\bf x}) = \begin{bmatrix}
			\frac{p^3}{p^2} \\
			\frac{q^3}{q^1} \\
		\end{bmatrix}.\end{equation}
	The map \eqref{figammalambda} plays an important role in the determination of the group   $\Gamma(1,-1)\cup\Lambda$ consisting of matrices 
	$\begin{bmatrix}
		\epsilon & {\bf 0}^t & 0\\
		\phi(X^{(\epsilon)},{\bf x}) & X^{(\epsilon)} & {\bf 0}\\
		u & {\bf x}^t & |X^{(\epsilon)}|
	\end{bmatrix}$ with $\epsilon=\pm 1$. 
	In the following we will denote this group by $(\Gamma(1,-1)\cup\Lambda)^\phi$. 
	The multiplication of $(\Gamma(1,-1)\cup\Lambda)^\phi$ has the form 
	\begin{equation}\label{szorzasmatrix}\begin{split}\begin{bmatrix}
				\epsilon & {\bf 0}^t & 0\\
				\phi(A^{(\epsilon)},{\bf a}) & A^{(\epsilon)} &  {\bf 0}\\
				u & {\bf a}^t & |A^{(\epsilon)}|
			\end{bmatrix}\begin{bmatrix}
				\delta & {\bf 0}^t & 0\\
				\phi(B^{(\delta)},{\bf b}) & 
				B^{(\delta)} &  {\bf 0}\\
				v & {\bf b}^t & |B^{(\delta)}|
			\end{bmatrix}=\quad&\\= \begin{bmatrix}
				\epsilon\delta & {\bf 0}^t & 0\\
				\phi(A^{(\epsilon)},{\bf a})\delta+A^{(\epsilon)}\phi(B^{(\delta)},{\bf b}) & A^{(\epsilon)}B^{(\delta)} &  {\bf 0} \\
				u\delta + |A^{(\epsilon)}|v + {\bf a}^t\phi(B^{(\delta)},{\bf b}) & {\bf a}^tB^{(\delta)} + |A^{(\epsilon)}|{\bf b}^t & |A^{(\epsilon)}B^{(\delta)}|
			\end{bmatrix}&.\end{split}\end{equation}
	Hence the map $\begin{bmatrix}
		\epsilon & {\bf 0}^t & 0\\
		\phi(X^{(\epsilon)},{\bf x}) & X^{(\epsilon)} &  {\bf 0}\\
		u & {\bf x}^t & |X^{(\epsilon)}|
	\end{bmatrix}\mapsto \left(\begin{bmatrix}
		X^{(\epsilon)} &  {\bf 0}\\
		{\bf x}^t & |X^{(\epsilon)}|
	\end{bmatrix},u\right)$ is an isomorphism of the group $(\Gamma(1,-1)\cup\Lambda)^\phi$ to the group extension defined by the multiplication 
	\begin{equation}\label{szorzasmatrixbovit}\begin{split}
			&\left(\mathcal{A}^{(\epsilon)},u \right) \left(\mathcal{B}^{(\delta)},v \right)= 
			\left(\mathcal{A}^{(\epsilon)} \mathcal{B}^{(\delta)}, u\delta + |A^{(\epsilon)}|v + {\bf a}^t\phi(B^{(\delta)},	{\bf b})\right), \end{split}\end{equation}
		where $\mathcal{A}^{(\epsilon)}= \begin{bmatrix}
			A^{(\epsilon)} & 0\\
			{\bf a}^t & |A^{(\epsilon)}|
		\end{bmatrix} $,  $\mathcal{B}^{(\delta)}= \begin{bmatrix}
		B^{(\delta)} & 0\\
		{\bf b}^t & |B^{(\delta)}|
	\end{bmatrix} $. 	
	Now, we can extend the assertion of Lemma \ref{lemiso} to the present situation:
	\begin{lem}\label{izomfi}
		We change the map $\phi$, defined by \eqref{figammalambda}, to $\lambda\phi$, for any $0\neq \lambda\in\mathbb K$ and denote the corresponding matrix group by $(\Gamma(1,-1)\cup\Lambda)^{\lambda\phi}$. The map $(\Gamma(1,-1)\cup\Lambda)^{\phi}\to(\Gamma(1,-1)\cup\Lambda)^{\lambda\phi}$ defined by $\begin{bmatrix}
			\epsilon & {\bf 0}^t & 0\\
			\phi(X^{(\epsilon)},{\bf x}) & X^{(\epsilon)} & {\bf 0} \\
			u & {\bf x}^t & |X^{(\epsilon)}|
		\end{bmatrix}\mapsto\begin{bmatrix}
			\epsilon & {\bf 0}^t & 0\\
			\lambda\phi(X^{(\epsilon)},{\bf x}) & X^{(\epsilon)} & {\bf 0} \\
			\lambda u & {\bf x}^t & |X^{(\epsilon)}|
		\end{bmatrix}$, with $\epsilon=\pm 1$, is a group isomorphism.
	\end{lem}
	\begin{proof}
		The assertion follows from the shape \eqref{szorzasmatrix} or \eqref{szorzasmatrixbovit} of the multiplication. 
	\end{proof}
	We return to the proof of the Proposition. Let $\sigma = \begin{bmatrix}
		-1& 0 & 0 & 0 \\
		0 &	0 & 1 & 0 \\
		0 &	1 & 0 & 0 \\
		0 &	0 & 0 & -1
	\end{bmatrix}$ be an involutive element of $\Lambda$. Since 
	$$\begin{bmatrix}
		-1& 0 & 0 & 0 \\
		0 &	0 & 1 & 0 \\
		0 &	1 & 0 & 0 \\
		0 &	0 & 0 & -1
	\end{bmatrix} \begin{bmatrix}
		1 &0 & 0 & 0 \\
		-\frac{q^3}{q^2} &p^1 & 0 & 0 \\
		-\frac{p^3}{p^1} &0 & q^2 & 0 \\
		u^3 &p^3 & q^3 & p^1q^2 
	\end{bmatrix}\begin{bmatrix}
		-1 & 0 & 0 & 0 \\
		0 &	0 & 1 & 0 \\
		0 &	1 & 0 & 0 \\
		0 &	0 & 0 & -1
	\end{bmatrix}=\begin{bmatrix}
		1 &0 & 0 & 0 \\
		\frac{p^3}{p^1} & q^2 & 0 & 0 \\
		\frac{q^3}{q^2} & 0 & p^1 & 0 \\
		u^3  & -q^3 & -p^3 & p^1q^2 
	\end{bmatrix},$$ 
	$\Gamma(1,-1)$ is a normal subgroup of index two of $(\Gamma(1,-1)\cup\Lambda)^{\phi}$.
	Moreover, the equation 
	\begin{equation}\label{mellekoszt} \begin{bmatrix}
			1 &0 & 0 & 0 \\
			-\frac{q^3}{q^2} &p^1 & 0	& 0 \\
			-\frac{p^3}{p^1} &0 & q^2 & 0 \\
			u^3 &p^3 & q^3 & p^1q^2 
		\end{bmatrix}\begin{bmatrix}
			-1& 0 & 0 & 0 \\
			0 &	0 & 1 & 0 \\
			0 &	1 & 0 & 0 \\
			0 &	0 & 0 & -1
		\end{bmatrix}=\begin{bmatrix}
			-1 &0 & 0 & 0 \\
			\frac{q^3}{q^2} & 0 & p^1 & 0 \\
			\frac{p^3}{p^1} & q^2 & 0 & 0 \\
			-u^3  & q^3 & p^3 & -p^1q^2 
		\end{bmatrix} \end{equation}
	shows that $\Lambda$ is the coset $\Gamma(1,-1)\cdot\sigma$ of $(\Gamma(1,-1)\cup\Lambda)^{\phi}$, hence $$\mathcal{\Aut}(\mathfrak{a}_1)= \Gamma(1,-1)\cup\Lambda=\Gamma(1,-1)\cup(\Gamma(1,-1)\cdot\sigma).$$ 
	Now, we apply the isomorphism described in Lemma \ref{izomfi} to the equation \eqref{mellekoszt} for $\lambda =-1$. Since this isomorphism preserves the involution $\sigma$,  we obtain the equation \eqref{Amatrixaut5uu}, proving the assertion.
\end{proof} 
Now we summarize the results of Section \ref{osztalyozas} about the group of automorphisms of the anti-commutative algebras given in Proposition \ref{anticommutativealgebras}.
\begin{prop}
	The automorphism groups of the algebras $\mathfrak{a}_0$, $\mathfrak{a}_1$, $\mathfrak{b}(\lambda)$, $\mathfrak{c}(\lambda)$, $\mathfrak{d}(\lambda,\mu)$ are isomorphic to the groups:
	\begin{equation*}\begin{split}
			\mathfrak{a}_0:\quad & \Gamma, \\
			\mathfrak{a}_1:\quad & \Gamma(-1,1)\cup(\Gamma(-1,1)\cdot\sigma), \\
			\mathfrak{b}(\lambda):\quad &\Gamma_\triangle,  \hspace{33pt} \\
			\mathfrak{c}(\lambda):\quad &\Gamma,\hspace{32pt}  \text{if}\quad\lambda\neq 1,\\
			&\Gamma_0 ,\hspace{28pt} \text{if}\quad\lambda=1, \\
			\mathfrak{d}(\lambda,\mu):\quad &\, \Gamma(1,1),\hspace{27pt}  \text{if}\quad \mu -2\lambda +1 \neq 0,\\
			&\, \Gamma(-1,1),\hspace{20pt}  \text{if}\quad \mu -2\lambda +1= 0.
	\end{split}\end{equation*}
\end{prop}

\bigskip
\noindent
Author's addresses: \\
\'Agota Figula, Institute of Mathematics, University of Debrecen, H-4002 Debrecen, P.O.Box 400, Hungary. {\it E-mail}: {\tt {}figula@science.unideb.hu} \\[1ex]
P\'eter T. Nagy, Institute of Applied Mathematics, \'Obuda University, 1034 Budapest, B\'ecsi \'ut 96/B, Hungary. {\it E-mail}: {\tt {}nagy.peter@nik.uni-obuda.hu}


\begin{thebibliography}{37}
	
	

	\bibitem{abdelwahab} H. Abdelwahab, A. J. Calderón Martín and A. Fernández Ouaridi,
	Central extensions of 4-dimensional binary Lie algebras,
	Rocky Mountain J. Math. 50, No. 5, (2020) 1541-1559.
	
	\bibitem{abdelwahab2} H. Abdelwahab, A. J. Calderón and I. Kaygorodov, The algebraic and geometric classification of nilpotent binary Lie algebras, International Jour. of Alg. and Comp. 29, No. 6, (2019)  1113-1129.
	

	
	\bibitem{BuSt}
	D. Burde, C. Steinhoff, Classification of orbit closures of 4-dimensional complex Lie algebras,
	J. Algebra 214 (1999) 729-739.
	
	\bibitem{capsto} S. Capozziello, C. Stornaiolo, Torsion tensor and its geometric interpretation, Annales de la Fondation Louis de Broglie, 32 (2007) 195--214.
	
	\bibitem{cedilnik} A. Cedilnik and M. Jerman, Classification of three-dimensional zeropotent
	algebras, International Electronic Journal of Algebra, 27 (2020) 127-146.
	

	
	\bibitem{elduque}
	A. Elduque, On Malcev modules, Comm. Algebra, 18, No. 5, (1990) 1551-1561.
	
	\bibitem{elduqueshestakov}
	A. Elduque and I. P. Shestakov, Irreducible non-Lie modules for Malcev superalgebras, J. Algebra,
	173, No. 3, (1995) 622-637.
	
	\bibitem{gainov0} A. Gainov, Identical relations for binary Lie rings, Uspehi Mat. Nauk N.S.,
	12, 3 (75)  (1957) 141-146. (Russian)
	
	
	\bibitem{gainov}
	A. T. Gainov, Binary Lie algebras of lower ranks, Algebra Logika, 2, No. 4,  (1963) 21-40.
	
	\bibitem{gavrilov}
	A. V. Gavrilov, Malcev Extensions, Southeast Asian Bulletin of Mathematics 34 (2010) 417-424.
	
	
	\bibitem{ismailov} N. Ismailov, I. Kaygorodov,  and Y. Volkov,  Degenerations of Leibniz and Anticommutative Algebras. Canadian Math. Bulletin, 62, No. 3, (2019) 539-549.
	
	
	\bibitem{kaykhrlop} I. Kaygorodov, M. Khrypchenko, S. A. Lopes, The algebraic and geometric classification of nilpotent anticommutative algebras, J. Pure Appl. Algebra 224:8 (2020) 106337.
	
	\bibitem{kaygorodov} I. Kaygorodov, Y. Popov, and Y. Volkov, Degenerations of binary Lie and nilpotent Malcev algebras, Comm. Algebra 46:11 (2018) 4928-4940.
	
	
	\bibitem{kobayashi} Yu. Kobayashi, K. Shirayanagi, S. Takahasi, and M. Tsukada, Classification of three-dimensional zeropotent algebras over an
	algebraically closed field, Comm. Algebra, 45:12 (2017) 5037-5052.
	
	\bibitem{kuzmin1}
	E. N. Kuzmin, Malcev algebras and their representations, Algebra and Logic 7 (1968) 48-69.
	
	\bibitem{kuzmin2}
	E. N. Kuzmin, Binary Lie algebras of small dimensions, Algebra and Logic 37 (1998) 181-186.
	
	
	\bibitem{kuzmin3} E. N. Kuzmin and I. P. Shestakov, Nonassociative structures, in Algebra 6, Itogi Nauki
	Tekh., Ser. Fund. Napr., 57, VINITI, Moscow (1990) 179-266.
	
	
    \bibitem{landsberg} J. M. Landsberg, Tensors: geometry and applications, Graduate Studies in Mathematics, vol. 128, American Mathematical Society, Providence, RI, (2012).
	
	
	\bibitem{malcev} A. I. Malcev, Analytic loops, Mat. Sb. N.S. 36(78), No. 3, (1955) 569-576. (Russian)
	
	\bibitem{petti} R. J. Petti, Derivation of Einstein--Cartan theory from general relativity, International Journal of Geometric Methods in Modern Physics 18, No. 06, 2150083 (2021). 
	
	\bibitem{sagle}
	A. A. Sagle, Malcev algebras, Trans. Amer. Math. Soc. 101 (l961)
	426-458.
	
	\bibitem{shirayanagi} K. Shirayanagi, S. Takahasi, M. Tsukada and Y. Kobayashi, Classification
	of three-dimensional zeropotent algebras over the real number field, Comm.
	Algebra, 46:11 (2018) 4663-4681.
	
	
\end{thebibliography}
\end{document}